\documentclass[12pt]{article}
\usepackage{amsmath, amssymb, amsthm, enumerate, booktabs, accents, framed, graphicx, array, color, multirow, mathrsfs, amsrefs, mathrsfs, cases, euscript, bm, multicol, caption, subfigure}

\usepackage[T1]{fontenc}
\usepackage[utf8]{inputenc}
\usepackage[colorlinks,linkcolor=blue,citecolor=blue,urlcolor=blue]{hyperref}
\usepackage[usenames,dvipsnames,svgnames,table]{xcolor}
\usepackage{a4wide}
\usepackage{tikz}

\allowdisplaybreaks[4]

\numberwithin{equation}{section} \theoremstyle{plain}
\newtheorem{maintheorem}{Theorem}
\newtheorem{theorem}{Theorem}[section]
\newtheorem{lemma}[theorem]{Lemma}
\newtheorem{corollary}[theorem]{Corollary}
\newtheorem{proposition}[theorem]{Proposition}

\newtheorem{assumption}{Assumption}

\newtheorem{definition}[theorem]{Definition}
\newtheorem{remark}[theorem]{Remark}

\def\bC{\mathbb C}
\def\bE{\mathbb E}
\def\bN{\mathbb N}

\def\cC{\mathcal C}
\def\cI{\mathcal I}
\def\cJ{\mathcal J}
\def\cP{\mathcal P}

\def\y{\mathbf y}

\def\Tr{\mathrm {Tr}}

\def\Var{\mathrm {Var}}
\def\deg{\mathrm {deg}}

\begin{document}

\title{Limiting moments of sample covariance matrices from large fold tensor vectors}

\date{\today}
\author{Wangjun Yuan\thanks{Department of Mathematics, Southern University of Science and Technology. E-mail: ywangjun@connect.hku.hk}}
\maketitle

\begin{abstract}
We study the moments of sample covariance matrices $M_{n,k,m}$ formed from independent $k$-fold tensor products of $n$-dimensional random vectors.
The base vectors have independent and identically distributed centered complex entries with unit variance and finite moments of all orders.
For every fixed $p \in \mathbb N_+$, under suitable moment conditions, we obtain an asymptotic formula for $p$th moment as $n, k, m$ tend to infinity with $m/n^k \to c \in (0,\infty)$.
There is no restriction on the growth rate of $k$ relative to $n$.
We also prove a variance bound and obtain the corresponding almost sure asymptotic formula.
The results recover known limiting moments in several regimes.
When $k/n$ tends to infinity and the modulus of base vector entry is not constant, we establish the asymptotic growth rate for the $p$th moment of normalized matrices.
In this regime, we also provide a lower bound for the growth rate of the largest eigenvalue.
\end{abstract}


\medskip 

 
\section{Introduction}
\label{sec:intro}

We write $\bN_+$ for the set of positive integers, and $\bN$ for the set of natural numbers. Let $n \in \bN_+$ and let $\{\xi_1, \ldots, \xi_n\}$ be a family of i.i.d. centered complex random variables with unit variance.
We consider the random vector $\y = (\xi_1, \ldots, \xi_n) \in \bC^n$, and let $\{\y_{\alpha}^{(l)}: 1 \le \alpha \le m, 1 \le l \le k\}$ be a family of i.i.d. copies of $\y$.
We use the notation $\otimes$ for the tensor outer product.
For $1 \le \alpha \le m$, we define $Y_{\alpha} = \y_{\alpha}^{(1)} \otimes \cdots \otimes \y_{\alpha}^{(k)}$ yielding a $k$-fold tensor product with entries
\begin{align*}
    \left( Y_\alpha \right)_{(j_1,\ldots,j_k)} = \prod_{l=1}^k \left( \y_\alpha^{(l)} \right)_{j_l}, \quad
    \forall 1 \le j_1, \ldots, j_k \le n.
\end{align*}
We refer readers to \cite{Tensorbook2018} for more details on tensor products.
We identify each $Y_{\alpha}$ as an $n^k$-dimensional vector, and write $Y = (Y_1, \ldots, Y_m)$. Let $\{\tau_1, \tau_2, \ldots\}$ be a sequence of nonnegative numbers. We consider the sum of $m$ independent rank-$1$ Hermitian matrices:
\begin{align}\label{eq:matrix}
	M_{n,k,m} = \sum_{\alpha=1}^m \tau_{\alpha} Y_{\alpha} Y_{\alpha}^*,
\end{align}
which is an $n^k \times n^k$ Hermitian matrix.

To compare with the existing literature, we also introduce the normalized vectors
\begin{align*}
    \widetilde{\y}
    &= \frac{\y}{\sqrt n} = \frac1{\sqrt n} (\xi_1, \ldots, \xi_n),
    \qquad
    \widetilde{\y}_\alpha^{(l)} = \frac{\y_\alpha^{(l)}}{\sqrt n}, \\
    \widetilde Y_\alpha
    &= \widetilde{\y}_\alpha^{(1)} \otimes \cdots \otimes \widetilde{\y}_\alpha^{(k)}
    = n^{-k/2} Y_\alpha,
\end{align*}
and the normalized matrix
\begin{align*}
    \widetilde M_{n,k,m}
    &= \sum_{\alpha=1}^m \tau_\alpha \widetilde Y_\alpha \widetilde Y_\alpha^*
    = n^{-k} M_{n,k,m}.
\end{align*}

When $k=1$, the matrix $M_{n,k,m}$ in \eqref{eq:matrix} is a sample covariance matrix with deterministic weights, which is an important tool for high dimensional data analysis in statistics.
The limiting spectral distribution (LSD) of such matrices is a classical topic in random matrix theory.
In particular, when $\tau_\alpha=1$ for all $\alpha$ and $m/n\to c\in(0,\infty)$, the empirical spectral distribution (ESD) of $\widetilde M_{n,1,m}$ converges to the Mar\v{c}enko-Pastur law under suitable moment conditions; see \cite{MP67}.
In the real i.i.d. case with $k=1$ and unit weights, Yin, Bai and Krishnaiah \cite{YinBaiKrishnaiah1988} proved that $\lambda_{\max}(\widetilde M_{n,1,m}) \to (1+\sqrt c)^2$ almost surely when $m/n \to c \in (0,\infty)$, provided the entries have a finite fourth moment.
We refer readers to \cites{Anderson2010,Bai2010} for further details.
For $k\ge2$, the tensor structure introduces dependence among the entries of each $Y_\alpha$, even though the base vectors have independent entries.
Thus, the usual assumptions for sample covariance matrices with independent entries do not apply directly.

Another motivation comes from quantum information theory, where tensor products describe product states.
Ambainis, Harrow and Hastings \cite{Hastings2012} studied sums of random product states and obtained the Mar\v{c}enko-Pastur moments for fixed $k$, with Gaussian base vectors or base vectors uniformly distributed on the unit sphere.
Their normalization corresponds to that of $\widetilde M_{n,k,m}$ above, with unit weights.
For fixed $k$ and unit weights, Ambainis, Harrow and Hastings \cite{Hastings2012}*{Corollary 3} also showed that the expected largest eigenvalue converges to $(1+\sqrt c)^2$ in their Gaussian and spherical models.
We would like to mention that $n$ is the dimension of each subsystem and $k$ is the number of subsystems.
Large values of $k$ are therefore natural for systems with many subsystems.
As discussed in \cite{BYY2022}, one would ideally keep $n$ fixed and let $k$ tend to infinity.
The regime $k \gg n$ is also motivated in \cite{Yuan2024}.
In the present paper, $n$ also tends to infinity, but we allow $k$ to grow arbitrarily fast relative to $n$.
Related random matrix models in quantum information theory were studied in \cites{Collins2011,Collins2013}.

Lytova \cite{Lytova2018} established the LSD under the same normalization for a class of real isotropic base vectors.
For independent entries with a finite fourth moment, her results also allows $k$ to grow with $n$ with the restriction $k = o(n)$; see \cite{Lytova2018}*{Theorem 1.2 and Remark 4.1}.
When $\tau_\alpha = 1$ for all $\alpha$, the limit is the Mar\v{c}enko-Pastur law.
Collins, Yao and Yuan \cite{BYY2022} considered the regime $k/n \to d \in [0,\infty)$ with $m/n^k \to c \in (0,\infty)$.
Assuming finiteness of all moments of the entries and convergence of the empirical moments of the weights, they computed the almost sure limit of the $p$th moment for $\widetilde M_{n,k,m}$ for all fixed $p \in \bN_+$.
Moreover, the limiting moments depend on $d$ and the fourth moment of the entries via $d(\bE [|\xi_1|^4] - 1)$; see \eqref{eq-example-linear-k-limit} below.
In the special case that $|\xi_1| = 1$ or $d=0$, for unit weights, the limiting moment sequence is exactly the Mar\v{c}enko-Pastur moment sequence.
When $d > 0$ and the modulus is not constant, the limiting moment sequence is different.
Yuan \cite{Yuan2024} focused on the setting of centered entries satisfying $|\xi_1| = 1$ almost surely and removed the restriction on the growth rate of $k$.
For $n, k \to \infty$ and $m/n^k \to c \in (0,\infty)$, for any fixed $p \in \bN_+$, \cite{Yuan2024}*{Theorems 1.1 and 1.2} give convergence of the $p$th moment of $\widetilde M_{n,k,m}$, in expectation and almost surely, under moment convergence of the weights.
With an additional growth bound on the limiting moments of the weights, the moment sequence determines a unique probability measure, which implies that the ESD of $\widetilde M_{n,k,m}$ converges to this measure almost surely; see \cite{Yuan2024}*{Corollary 1.1}.
In particular, the limit is the Mar\v{c}enko-Pastur law when all weights are one.
In a setting similar to that of \cite{Yuan2024}, Yaskov \cite{Yaskov2025} gave another proof of almost sure weak convergence of the ESD of $\widetilde M_{n,k,m}$ for arbitrary $k = k(n)$.

Let's mention some literature on spectrum of the tensor model.
Fluctuations of related tensor matrix models have also been studied for $k = 2$ in \cites{DembczakLytova2023,Dembczak2026}.
For the corresponding sample correlation matrices, Yuan \cite{Yuan2026correlation} proved convergence in probability of the ESD when $n,k \to \infty$, $k=o(n)$, $m/n^k \to c \in (0,\infty)$.

We impose the following assumption on the distribution of $\xi_1$.
\begin{assumption} \label{Assumption-distribution}
The distribution of $\xi_1$ does not depend on $n$, and $m_p = \bE [|\xi_1|^p] < \infty$ for every $p \in \bN_+$. We also set $m_0=1$.
\end{assumption}
We also present the following two assumptions on the mixed moments of $\xi_1$ and $\overline{\xi_1}$.
\begin{assumption} \label{Assumption-mixed-moments}
For all $p,q \in \bN_+$ with $|p-q|=1$, we have $\bE[\xi_1^p\overline{\xi_1}^q] = 0$.
\end{assumption}

\begingroup
\addtocounter{assumption}{-1}
\renewcommand{\theassumption}{\arabic{assumption}$'$}
\def\theHassumption{\arabic{assumption}prime}
\begin{assumption} \label{Assumption-distribution-prime}
For all nonnegative integers $p,q$ with $p\ne q$, we have $\bE[\xi_1^p\overline{\xi_1}^q] = 0$.
\end{assumption}
\endgroup

\begin{remark}
It is easy to see that Assumption \ref{Assumption-distribution-prime} is stronger than Assumption \ref{Assumption-mixed-moments}.
\end{remark}

For the dimension parameters $m,n,k$, we consider the following limiting setting.
\begin{assumption} \label{Assumption-ratio}
    $n \to \infty$, $k=k(n) \to \infty$, and $m=m(n) \to \infty$ such that
    \begin{align*}
        \dfrac{m}{n^k} \to c \in (0,\infty).
    \end{align*}
\end{assumption}
We would like to remark that there is no restriction on the growth rate of $k$ relative to $n$.

For nonnegative integers $n$ and $r$, we write
\begin{align*}
    (n)_r = n(n-1) \cdots(n-r+1)
\end{align*}
with the convention $(n)_0 = 1$.
For $d \in \bN$, let $\cP(d)$ be the set of all set partitions of $[d]=\{1,\ldots,d\}$, with the conventions $[0]=\varnothing$ and $\cP(0)=\{\varnothing\}$. Define
\begin{align} \label{eq-H-explicit}
    H_n(d)=\sum_{\pi\in\cP(d)}(n)_{|\pi|}
    \prod_{B\in\pi}m_{2|B|}, \qquad n \ge 1, d \ge 0,
\end{align}
where $|\pi|$ is the number of blocks of $\pi$, $|B|$ is the size of the block $B$, and an empty product is interpreted as one.
Moreover, for every fixed $d\ge1$, it is easy to see that $H_n(d)$ is a polynomial in $n$ of degree $d$ with leading coefficient one.
We refer to Lemma \ref{lem-norm-moment-polynomial} for an equivalent definition of $H_n(d)$ under Assumption \ref{Assumption-distribution}.

The following theorem, which is the first main result, gives an asymptotic formula for every fixed-order moment of $M_{n,k,m}$ in expectation.

\begin{maintheorem} \label{Thm-main}
Let Assumptions \ref{Assumption-distribution}, \ref{Assumption-mixed-moments}, and \ref{Assumption-ratio} hold.
Suppose that
\begin{align} \label{eq-tau-moment-convergence}
    \frac1m \sum_{j=1}^m \tau_j^q \longrightarrow m_q^{(\tau)} \in (0,+\infty),
    \qquad m \to \infty, \quad q\in \bN_+.
\end{align}
Then, for every fixed $p\in\bN_+$,
\begin{align} \label{eq-main-asymptotic}
    \frac1{n^k}\bE\left[\Tr M_{n,k,m}^p\right]
    =&\Bigg[ \sum_{s=1}^p c^s \sum_{\alpha \in \cC_{s,p}^{(1)}} \left( \prod_{a=1}^s m_{\deg_a(\alpha)}^{(\tau)} \right) \left( \prod_{a=1}^s H_n(\deg_a(\alpha)) \right)^k \Bigg](1+o(1)).
\end{align}
\end{maintheorem}

Here $\cC_{s,p}^{(1)}$ is the set of canonical non-crossing sequences of length $p$ with $s$ distinct values, and $\deg_a(\alpha)$ is the number of occurrences of $a$ in $\alpha$. These notions are introduced in Section \ref{sec:graph}.

\begin{remark} \label{rmk-H-degree-recursion}
For each fixed positive integer $n$, the quantities $H_n(d)$ can also be computed recursively in $d$ via
\begin{align} \label{eq-H-degree-recursion}
    H_n(d)
    = \sum_{\ell=1}^d \left[ n \binom{d-1}{\ell-1} - \binom{d-1}{\ell} \right] m_{2\ell} H_n(d-\ell), \qquad d\ge1
\end{align}
with the initial condition $H_n(0) = 1$ for $n \ge 1$.
Here $\binom ab=0$ whenever $b<0$ or $b>a$, for nonnegative integers $a$. In particular, $H_n(1)=n$. The recurrence is proved in Lemma \ref{lem-norm-moment-polynomial}.
\end{remark}

We also establish the following variance estimate, which yields the corresponding almost sure asymptotic formula in Corollary \ref{cor-normalized-moment-as}.

\begin{maintheorem} \label{Thm-variance}
Let Assumptions \ref{Assumption-distribution} and \ref{Assumption-ratio} hold, and suppose that \eqref{eq-tau-moment-convergence} holds. Then, for every fixed $p \in \bN_+$, there is a finite constant $C_p \ge 1$ such that, for all sufficiently large $n$,
\begin{align} \label{eq-main-variance}
    \Var \left( \frac1{n^k} \Tr M_{n,k,m}^p \right)
    &\le C_p^k n^{(2p-1)k}.
\end{align}
\end{maintheorem}

Next, we turn to the normalized matrix $\widetilde M_{n,k,m}$. Under Assumption \ref{Assumption-ratio} and \eqref{eq-tau-moment-convergence}, the first moment satisfies
\begin{align} \label{eq-example-first-moment}
    \frac1{n^k} \bE [\Tr \widetilde M_{n,k,m}]
    &= \frac1{n^k} \sum_{\alpha=1}^m \tau_\alpha
    = \frac{m}{n^k} \left( \frac1m \sum_{\alpha=1}^m \tau_\alpha \right)
    \longrightarrow c m_1^{(\tau)} > 0.
\end{align}
Thus, this normalization gives a finite positive limiting first spectral moment.
In particular, we believe that the normalization $\widetilde M_{n,k,m}$ is appropriate.
However, the higher moment of $\widetilde M_{n,k,m}$ is exploded when $m_4>1$ and $k/n\to\infty$. See Proposition \ref{prop-example-moment-growth}. In this case, we use the asymptotic moment formula to derive the lower bounds for the growth of the largest eigenvalue
\begin{align*}
    \frac nk\log\lambda_{\max}(\widetilde M_{n,k,m})
    &\longrightarrow+\infty \qquad \text{almost surely}.
\end{align*}
We would like to point out that the growth rate is not optimal. The optimal growth rate will be studied as another project.

Our proof uses the method of moments and the graph representation developed in \cites{Bai2010,BYY2022,Yuan2024}.
In the moment expansion, the tensor structure raises each single-factor contribution to the power $k$.
When $k=o(n)$, its leading term is enough to determine the limit.
When $k$ is of order $n$, the second leading term can also contribute.
For unrestricted $k$, all the terms have contribution before taking the $k$th power. We count all the terms explicitly.
This is the main novelty of the paper.

The rest of the paper is organized as follows. We develop the theory of graph combinatorics in Section \ref{sec:graph}.
We compute the moments exactly in Section \ref{sec:moment}.
In Section \ref{sec:asymptotic-moment}, we derive the asymptotic moment formula and prove Theorem \ref{Thm-main}.
In Section \ref{sec:variance}, we prove Theorem \ref{Thm-variance}.
Section \ref{sec:examples} recovers the known moment formulas in the unit-modulus, $k=o(n)$, and $k/n\to d\in(0,\infty)$ regimes. In Section \ref{sec:example-quadratic-k}, we establish the moment growth formulas when $k/n\to\infty$ and derive lower bounds for the largest eigenvalue from these estimates.

Throughout the paper, we assume that Assumption \ref{Assumption-distribution} holds. For the spectral moment calculation, we first provide the proofs under Assumption \ref{Assumption-mixed-moments}; whenever Assumption \ref{Assumption-distribution-prime} allows for a shorter argument, we indicate it separately.

\section{Graph combinatorics} \label{sec:graph}

In this section, we introduce the graph combinatorics that will be used later.
For the graph and moment calculations in this section and Section \ref{sec:moment}, we fix $p$ and take $n,m \ge p$, which holds for all sufficiently large $n$ and $m$.

\subsection{Graphs for moments} \label{sec:graph-def}

In this subsection, we introduce the graph that corresponds to the moment following the framework of \cites{Bai2010, BYY2022,Yuan2024}.

We call $\alpha=(\alpha_1, \ldots, \alpha_p) \in [m]^p$ a sequence of length $p$.
We denote by $|\alpha|$ the number of distinct elements in $\alpha$. If $s = |\alpha|$, then we call $\alpha$ an \emph{$s$-sequence}. Let  $\cJ_{s,p}(m)$ be the set of all $s$-sequences $\alpha \in [m]^p$. Then
\begin{align} \label{eq-0.3}
	[m]^p = \bigcup_{s=1}^p \cJ_{s,p}(m).
\end{align}
For a sequence $\alpha = (\alpha_1, \ldots, \alpha_p)$ and each value $t$ in $\alpha$, we count its frequency by
\begin{align}\label{eq:deg_t}
  \deg_t(\alpha) = \# \{j \in [p]: \alpha_j = t \},
\end{align}
where we use the notation $\#S$ for the number of elements in the set $S$.

Now we introduce the graph for the moment. For $\alpha = (\alpha_1, \ldots, \alpha_p) \in [m]^p$ and $i = (i^{(1)}, \ldots, i^{(p)}) \in [n]^p$, we draw two parallel lines referred to as the \emph{$\alpha$-line} and the \emph{$i$-line}, respectively.
Plot $i^{(1)}, \ldots, i^{(p)}$ on the $i$-line and $\alpha_1, \ldots, \alpha_p$ on the $\alpha$-line.
We call $i^{(1)}, \ldots, i^{(p)}$ and $\alpha_1, \ldots, \alpha_p$ the $i$-vertices and the $\alpha$-vertices.
Draw an edge from $\alpha_u$ to $i^{(u)}$ and call it a \emph{down edge} for $1 \le u \le p$. We also draw an edge from $i^{(u)}$ to $\alpha_{u+1}$ for $1 \le u \le p$, { where $\alpha_{p+1}:=\alpha_1$}, and the edges are called \emph{up edges}.
We identify vertices on the same line if and only if their values agree. Thus, the graph has one $\alpha$-vertex for each distinct value in $\alpha$, and one $i$-vertex for each distinct value in $i$. Vertices on different lines remain distinct even when they have the same value.
We denote the directed graph by $g(i,\alpha)$.
When indicating a vertex, we use its value. We may abuse the notation by using the component $\alpha_u$ or $i^{(u)}$ of the sequence $\alpha$ or $i$ if we want to emphasize the order of the visit of the particular vertex.

For a given sequence $\alpha$, we set $\Delta(p;\alpha)$ to be the collection of all directed graphs $g(i,\alpha)$, and we call them \emph{$\Delta(p;\alpha)$-graphs}. We provide an example of a $\Delta(p;\alpha)$-graph in Figure \ref{fig-delta-graph} below.
\begin{figure}[htbp]
\centering
\begin{tikzpicture}[line width=0.65pt,>=latex]
  \draw[dotted] (-0.45,2) -- (3.5,2);
  \draw[dotted] (-0.45,0) -- (3.5,0);
  \node[anchor=east] at (-0.55,2) {$\alpha$-line};
  \node[anchor=east] at (-0.55,0) {$i$-line};
  \node[circle,fill=black,inner sep=2pt] (a1) at (0,2) {};
  \node[circle,fill=black,inner sep=2pt] (a2) at (3,2) {};
  \node[circle,fill=black,inner sep=2pt] (b1) at (0,0) {};
  \node[circle,fill=black,inner sep=2pt] (b2) at (3,0) {};
  \node[above=4pt] at (a1) {$\alpha_1=1$};
  \node[above=4pt] at (a2) {$\alpha_2=\alpha_3=2$};
  \node[below=4pt] at (b1) {$i^{(1)}=i^{(3)}=1$};
  \node[below=4pt] at (b2) {$i^{(2)}=2$};
  \draw[->] (a1) to[bend left=13] (b1);
  \draw[->] (a2) to[bend left=13] (b2);
  \draw[->] (a2) to[bend left=10] (b1);
  \draw[->] (b1) to[bend left=10] (a2);
  \draw[->] (b2) to[bend left=13] (a2);
  \draw[->] (b1) to[bend left=13] (a1);
\end{tikzpicture}
\caption{A $\Delta(p;\alpha)$-graph with $p=3$, $\alpha=(1,2,2)$, and $i=(1,2,1)$.}
\label{fig-delta-graph}
\end{figure}
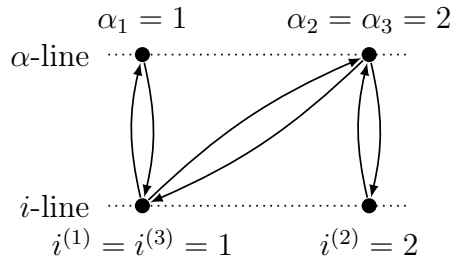

From the definition, one can easily see that the directed graph $g(i,\alpha)$ is connected, with $p$ up edges and $p$ down edges appearing alternately.
If there are edges that have the same endpoints, then they are \emph{coincident}. If coincident edges have the same direction, then the directed graph has multiple edges.
Let $1 \le a \le m$ and $1 \le b \le n$. Let $e_d(a,b)$ and $e_u(a,b)$ denote the number of multiple down edges from the $\alpha$-vertex with value $a$ to the $i$-vertex with value $b$, and the number of multiple up edges from the $i$-vertex with value $b$ to the $\alpha$-vertex with value $a$, respectively. We use the convention that $e_d(a,b) = e_u(a,b) = 0$ if either none of the $\alpha$-vertices takes value $a$ or none of the $i$-vertices takes value $b$.

The i.i.d. setting of $\xi_i$ implies that the moment depends on the graph rather than the values of the vertices. To simplify the sum, we introduce the equivalence of sequences and graphs.
Two sequences $\alpha=(\alpha_1, \ldots, \alpha_p), \beta=(\beta_1, \ldots, \beta_p)$ in $[m]^p$ are \emph{equivalent} if there exists a permutation $\pi$ on $[m]$ such that $\beta_i = \pi(\alpha_i)$ for all $1 \le i \le p$.
The sequence $\alpha$ is \emph{canonical} if $\alpha_1=1$ and $\alpha_u \le \max\{\alpha_1, \ldots, \alpha_{u-1}\} + 1$ for $u \ge 2$. We denote by $\cC_{s,p}$ the set of all canonical $s$-sequences of length $p$.
From the definition above, one can see that the set of distinct values of a canonical $s$-sequence is $[s]$. Denote by $\cI_{s,m}$ the set of injective maps from $[s]$ to $[m]$. For a canonical $s$-sequence $\alpha$ and a map $\varphi \in \cI_{s,m}$, we call $\varphi(\alpha)$ the $s$-sequence $(\varphi(\alpha_1), \ldots, \varphi(\alpha_p))$. For each canonical $s$-sequence, its image under the maps in $\cI_{s,m}$ gives all its equivalent sequences, and hence its equivalence class of sequences in $[m]^p$ has exactly $m(m-1) \cdots (m-s+1)$ distinct elements.

Two graphs $g(i,\alpha)$ and $g(i',\alpha')$ are called \emph{equivalent}, if the two sequences $i$ and $i'$ are equivalent, and the two sequences $\alpha$ and $\alpha'$ are equivalent. In this case, we write $g(i,\alpha) \sim g(i',\alpha')$ for this equivalence.
For each equivalence class, we may choose the canonical graph such that $\alpha = (\alpha_1, \ldots, \alpha_p) \in [m]^p$ is a canonical $s$-sequence and $i = (i^{(1)}, \ldots, i^{(p)}) \in [n]^p$ is a canonical $r$-sequence, for some $r,s \in \bN_+$.

We end this subsection by introducing two operations on the $\Delta(p;\alpha)$-graph $g(i,\alpha)$ for a canonical $r$-sequence $i$ of length $p$ and a canonical $s$-sequence $\alpha$ of length $p$. The operations will be used later.

The first operation is the \emph{splitting of an $i$-vertex}.
Let $v$ be an $i$-vertex of $g(i,\alpha)$, and denote by
$I=\{u\in[p]: i^{(u)}=v\}$ the set of positions in the sequence $i$ whose value equals $v$.
We define the \emph{splitting of the $i$-vertex $v$} as follows.
Let $J$ be a non-empty proper subset of $I$.
We replace $v$ by two distinct $i$-vertices $v_1$ and $v_2$, whose values are different from those of all the other $i$-vertices.
For every $u\in J$, the directed edges $\alpha_u \longrightarrow i^{(u)}$ and $i^{(u)} \longrightarrow \alpha_{u+1}$ are replaced by $\alpha_u \longrightarrow v_1$ and $v_1 \longrightarrow \alpha_{u+1}$, respectively.
For every $u\in I \setminus J$, they are replaced by $\alpha_u \longrightarrow v_2$ and $v_2 \longrightarrow \alpha_{u+1}$, respectively.
All the other vertices and edges remain unchanged.
The resulting graph has $r+1$ distinct $i$-vertices.
Whenever necessary, we canonically relabel the $i$-vertices of the resulting graph.
An example is shown in Figure \ref{fig-splitting-i-vertex}.
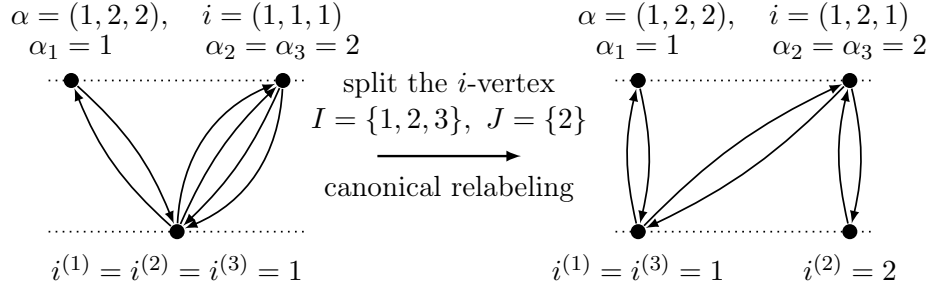
\begin{figure}[htbp]
\centering
\begin{tikzpicture}[line width=0.65pt,>=latex,font=\small]
  \node at (1.4,2.85) {$\alpha=(1,2,2),\quad i=(1,1,1)$};
  \node at (8.9,2.85) {$\alpha=(1,2,2),\quad i=(1,2,1)$};
  \draw[dotted] (-0.3,2) -- (3.1,2);
  \draw[dotted] (-0.3,0) -- (3.1,0);
  \node[circle,fill=black,inner sep=2pt] (a1) at (0,2) {};
  \node[circle,fill=black,inner sep=2pt] (a2) at (2.8,2) {};
  \node[circle,fill=black,inner sep=2pt] (b1) at (1.4,0) {};
  \node[above=4pt] at (a1) {$\alpha_1=1$};
  \node[above=4pt] at (a2) {$\alpha_2=\alpha_3=2$};
  \node[below=4pt] at (b1) {$i^{(1)}=i^{(2)}=i^{(3)}=1$};
  \draw[->] (a1) to[bend left=13] (b1);
  \draw[->] (a2) to[bend left=11] (b1);
  \draw[->] (a2) to[bend left=32] (b1);
  \draw[->] (b1) to[bend left=11] (a2);
  \draw[->] (b1) to[bend left=32] (a2);
  \draw[->] (b1) to[bend left=13] (a1);
  \draw[->,line width=0.9pt] (4.05,1) -- (5.95,1);
  \node[align=center,above] at (5,1.13) {split the $i$-vertex\\$I=\{1,2,3\},\ J=\{2\}$};
  \node[align=center,below] at (5,0.88) {canonical relabeling};
  \begin{scope}[xshift=7.5cm]
    \draw[dotted] (-0.3,2) -- (3.1,2);
    \draw[dotted] (-0.3,0) -- (3.1,0);
    \node[circle,fill=black,inner sep=2pt] (c1) at (0,2) {};
    \node[circle,fill=black,inner sep=2pt] (c2) at (2.8,2) {};
    \node[circle,fill=black,inner sep=2pt] (d1) at (0,0) {};
    \node[circle,fill=black,inner sep=2pt] (d2) at (2.8,0) {};
    \node[above=4pt] at (c1) {$\alpha_1=1$};
    \node[above=4pt] at (c2) {$\alpha_2=\alpha_3=2$};
    \node[below=4pt] at (d1) {$i^{(1)}=i^{(3)}=1$};
    \node[below=4pt] at (d2) {$i^{(2)}=2$};
    \draw[->] (c1) to[bend left=13] (d1);
    \draw[->] (c2) to[bend left=13] (d2);
    \draw[->] (c2) to[bend left=10] (d1);
    \draw[->] (d1) to[bend left=10] (c2);
    \draw[->] (d2) to[bend left=13] (c2);
    \draw[->] (d1) to[bend left=13] (c1);
  \end{scope}
\end{tikzpicture}
\caption{Splitting the $i$-vertex of value $1$ according to $J=\{2\}$.}
\label{fig-splitting-i-vertex}
\end{figure}

The second one is the
\emph{gluing of $i$-vertices}.
For $2\le h\le r$, let $v_1,\ldots,v_h$ be $h$ distinct $i$-vertices of $g(i,\alpha)$.
We define the \emph{gluing of the $i$-vertices $v_1,\ldots,v_h$} as follows.
We replace $v_1,\ldots,v_h$ by a single $i$-vertex $v$.
For every $u \in [p]$ such that the $i$-vertex $i^{(u)}$ is one of $v_1, \ldots, v_h$, the edges $\alpha_u \longrightarrow i^{(u)}$ and $i^{(u)} \longrightarrow \alpha_{u+1}$ are replaced by $\alpha_u \longrightarrow v$ and $v \longrightarrow \alpha_{u+1}$, respectively.
All the other vertices and edges remain unchanged.
In particular, distinct edges before the gluing remain distinct after the gluing, although some of them may become coincident.
The resulting graph has $r-h+1$ distinct $i$-vertices.
Whenever necessary, we canonically relabel the $i$-vertices of the resulting graph.
We provide an example of gluing of two $i$-vertices in Figure \ref{fig-gluing-i-vertices} below.

\begin{figure}[htbp]
\centering
\begin{tikzpicture}[line width=0.65pt,>=latex,font=\small]
  \node at (1.4,2.85) {$\alpha=(1,2,2),\quad i=(1,2,3)$};
  \node at (8.9,2.85) {$\alpha=(1,2,2),\quad i=(1,2,1)$};
  \draw[dotted] (-0.3,2) -- (3.1,2);
  \draw[dotted] (-0.3,0) -- (3.1,0);
  \node[circle,fill=black,inner sep=2pt] (a1) at (0,2) {};
  \node[circle,fill=black,inner sep=2pt] (a2) at (2.8,2) {};
  \node[circle,fill=black,inner sep=2pt] (b1) at (0,0) {};
  \node[circle,fill=black,inner sep=2pt] (b2) at (1.4,0) {};
  \node[circle,fill=black,inner sep=2pt] (b3) at (2.8,0) {};
  \node[above=4pt] at (a1) {$\alpha_1=1$};
  \node[above=4pt] at (a2) {$\alpha_2=\alpha_3=2$};
  \node[below=4pt] at (b1) {$i^{(1)}=1$};
  \node[below=4pt] at (b2) {$i^{(2)}=2$};
  \node[below=4pt] at (b3) {$i^{(3)}=3$};
  \draw[->] (a1) to[bend right=10] (b1);
  \draw[->] (a2) to[bend left=13] (b2);
  \draw[->] (a2) to[bend left=10] (b3);
  \draw[->] (b1) -- (a2);
  \draw[->] (b2) to[bend left=13] (a2);
  \draw[->] (b3) -- (a1);
  \draw[->,line width=0.9pt] (4.05,1) -- (5.95,1);
  \node[align=center,above] at (5,1.13) {glue the $i$-vertices\\of values $1$ and $3$};
  \node[align=center,below] at (5,0.88) {canonical relabeling};
  \begin{scope}[xshift=7.5cm]
    \draw[dotted] (-0.3,2) -- (3.1,2);
    \draw[dotted] (-0.3,0) -- (3.1,0);
    \node[circle,fill=black,inner sep=2pt] (c1) at (0,2) {};
    \node[circle,fill=black,inner sep=2pt] (c2) at (2.8,2) {};
    \node[circle,fill=black,inner sep=2pt] (d1) at (0,0) {};
    \node[circle,fill=black,inner sep=2pt] (d2) at (2.8,0) {};
    \node[above=4pt] at (c1) {$\alpha_1=1$};
    \node[above=4pt] at (c2) {$\alpha_2=\alpha_3=2$};
    \node[below=4pt] at (d1) {$i^{(1)}=i^{(3)}=1$};
    \node[below=4pt] at (d2) {$i^{(2)}=2$};
    \draw[->] (c1) to[bend left=13] (d1);
    \draw[->] (c2) to[bend left=13] (d2);
    \draw[->] (c2) to[bend left=10] (d1);
    \draw[->] (d1) to[bend left=10] (c2);
    \draw[->] (d2) to[bend left=13] (c2);
    \draw[->] (d1) to[bend left=13] (c1);
  \end{scope}
\end{tikzpicture}
\caption{Gluing the two $i$-vertices of values $1$ and $3$.}
\label{fig-gluing-i-vertices}
\end{figure}
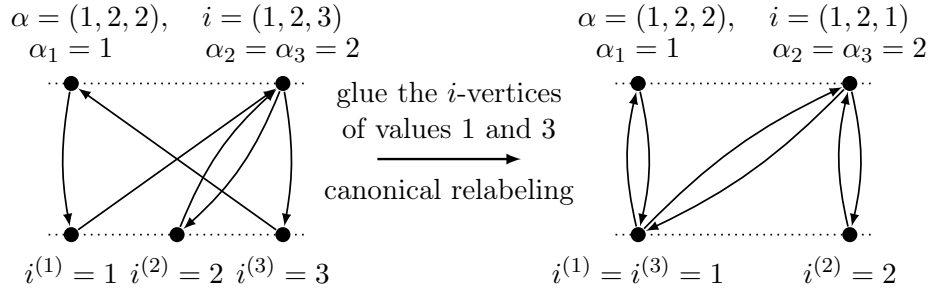

More generally, let $\pi=\{B_1,\ldots,B_l\}$ be a partition of the set
of distinct $i$-vertices of $g(i,\alpha)$.
We say that the resulting graph is obtained by
\emph{gluing the $i$-vertices according to $\pi$} if, for every
$1\le j\le l$, all the $i$-vertices belonging to $B_j$ are glued into
one $i$-vertex.
The resulting graph has exactly $l=|\pi|$ distinct $i$-vertices.

\begin{remark}
The gluing of $i$-vertices is the reverse operation of the splitting of an $i$-vertex.
More precisely, if an $i$-vertex is split into two $i$-vertices, then gluing these two $i$-vertices recovers the original graph.
Conversely, if several $i$-vertices are glued into one $i$-vertex, the original graph can be recovered by successively splitting the glued $i$-vertex and assigning the corresponding pairs of down and up edges to the original $i$-vertices.
\end{remark}

\begin{remark}
Any graph obtained from a $\Delta(p;\alpha)$-graph by gluing or splitting $i$-vertices is still a $\Delta(p;\alpha)$-graph.
Indeed, neither operation changes the cyclic order of the $p$ pairs of down and up edges, and the resulting graph remains connected with $p$ down edges and $p$ up edges appearing alternately.    
\end{remark}

\subsection{Classification of graphs} \label{sec:graph-classify}

In this subsection, we classify the $\Delta(p;\alpha)$-graphs defined in Subsection \ref{sec:graph-def}.

For any given sequence $\alpha = (\alpha_1, \ldots, \alpha_p) \in [m]^p$, we classify the graphs in $\Delta(p;\alpha)$ as follows.
The set of graphs $\Delta^{(1)}(p;\alpha)$ consists of all $\Delta(p;\alpha)$-graphs $g(i,\alpha)$ such that $e_u(a,b) = e_d(a,b)$ for all $1 \le a \le m$ and $1 \le b \le n$.
In this case, we write $e(a,b) = e_u(a,b)$.
We denote by $\Delta^{(2)}(p;\alpha)$ the remaining $\Delta(p;\alpha)$-graphs.

The following remark contains two observations about the operations introduced in Subsection \ref{sec:graph-def}.

\begin{remark} \label{Rmk-glue-split-Delta(1)}
Suppose that $g(i,\alpha)\in\Delta^{(1)}(p;\alpha)$.
\begin{enumerate}
    \item Then any graph obtained from $g(i,\alpha)$ by gluing $i$-vertices also belongs to $\Delta^{(1)}(p;\alpha)$.
    Indeed, suppose that $v_1,\ldots,v_h$ are glued into a single $i$-vertex $v$. For every value $a$ occurring in $\alpha$, since $e_d(a,v_j)=e_u(a,v_j)$ for all $1\le j\le h$, we obtain that
    \begin{align*}
        e_d(a,v)
        = \sum_{j=1}^h e_d(a,v_j)
        = \sum_{j=1}^h e_u(a,v_j)
        = e_u(a,v).
    \end{align*}
    Therefore, gluing $i$-vertices preserves the defining property of $\Delta^{(1)}(p;\alpha)$.
    \item A graph obtained by splitting an $i$-vertex does not necessarily belong to $\Delta^{(1)}(p;\alpha)$.
    More precisely, suppose that an $i$-vertex $v$ is split into $v_1$ and $v_2$ according to a non-empty { proper} subset $J\subset I$.
    Then the resulting graph belongs to $\Delta^{(1)}(p;\alpha)$ if and only if, { for every value $a$ occurring in $\alpha$},
    \begin{align*}
        \#\{u\in J:\alpha_u=a\}
        =
        \#\{u\in J:\alpha_{u+1}=a\}.
    \end{align*}
    Indeed, the two sides are exactly $e_d(a,v_1)$ and $e_u(a,v_1)$, respectively. Since the original vertex $v$ satisfies $e_d(a,v)=e_u(a,v)$, the equality for $v_1$ holds if and only if it holds for $v_2$; all other $i$-vertices are unchanged.
\end{enumerate}
\end{remark}

Next, we introduce two operations for graphs belonging to $\Delta^{(1)}(p;\alpha)$.
\begin{definition}
Let $\alpha\in [m]^p$ and $i \in [n]^p$ be two arbitrary sequences such that the graph $g(i,\alpha) \in \Delta^{(1)}(p;\alpha)$.
For any $1 \le a \le m$ and $1 \le b \le n$, for the $\alpha$-vertex with value $a$ and $i$-vertex with value $b$, we replace all the directed edges (including up edges and down edges) between the two vertices by $e(a,b)$ many non-directed edges. We call the resulting graph the \emph{reduced graph} of $g(i,\alpha)$ and denote it by $g^{red}(i,\alpha)$.
\end{definition}

\begin{definition}
Let $\alpha\in [m]^p$ and $i \in [n]^p$ be two arbitrary sequences such that the graph $g(i,\alpha) \in \Delta^{(1)}(p;\alpha)$.
For any $1 \le a \le m$ and $1 \le b \le n$, for the $\alpha$-vertex with value $a$ and $i$-vertex with value $b$, if $e(a,b)>0$, then we replace all the coincident directed edges between the two vertices with one non-directed edge. We call the new graph the \emph{simplification} of $g(i,\alpha)$ and we use the notation $g^{simp}(i,\alpha)$.
\end{definition}

\begin{remark} \label{Rmk-red-simp-properties}
The graphs $g^{red}(i,\alpha)$, $g^{simp}(i,\alpha)$ and $g(i,\alpha)$ have the same set of vertices.
The graph $g^{red}(i,\alpha)$ is connected with exactly $p$ edges.
The graph $g^{simp}(i,\alpha)$ is connected without multiple edges.
\end{remark}

For a given sequence $\alpha\in[m]^p$, we denote by $\Delta^{(1)}_0(p;\alpha)$ the collection of all graphs $g(i,\alpha)\in\Delta^{(1)}(p;\alpha)$ such that $g^{red}(i,\alpha)$ does not have multiple edges and $g^{simp}(i,\alpha)$ does not contain any cycles. For every such graph,
\begin{align*}
    g^{red}(i,\alpha) = g^{simp}(i,\alpha).
\end{align*}
By Remark \ref{Rmk-red-simp-properties}, this graph is connected, has exactly $p$ edges, and has the same vertices as $g(i,\alpha)$. Moreover, $g^{red}(i,\alpha)$ is a tree, and hence $g(i,\alpha)$ has exactly $p+1$ distinct vertices. These graphs correspond to the leading terms in the moment expansion. We collect some results on them from \cites{BYY2022,Yuan2024}.

For a given sequence $\alpha \in \cC_{s,p}$, the following lemma counts the number of canonical sequences $i \in \cC_{p+1-s,p}$ such that $g(i,\alpha) \in {\Delta^{(1)}_0(p;\alpha)}$.

\begin{lemma} {(\cite{BYY2022}*{Lemma 3.1})} \label{lem-Bai}
For any $1 \le s \le p$ and any sequence $\alpha \in \cC_{s,p}$, there is at most one sequence $i \in \cC_{p+1-s,p}$ such that $g(i,\alpha) \in {\Delta^{(1)}_0(p;\alpha)}$. We denote by $\cC_{s,p}^{(1)}$ the set of such canonical sequences $\alpha$. Then the number of elements in $\cC_{s,p}^{(1)}$ is
\begin{align*}
	\dfrac{1}{p} \binom{p}{s-1} \binom{p}{s}.
\end{align*}
\end{lemma}

In order to characterize the sequence in $\cC_{s,p}^{(1)}$, we require the following definition.

\begin{definition} \label{def-non-crossing}
A sequence $\alpha=(\alpha_1, \ldots, \alpha_p)$ is called a \emph{crossing} sequence if there exist $j_1<j_2<j_3<j_4$, such that $\alpha_{j_1} = \alpha_{j_3} \neq \alpha_{j_2} = \alpha_{j_4}$.
A sequence is called a \emph{non-crossing} sequence if it is not a crossing sequence.
\end{definition}

The next lemma characterizes the sequence in $\cC_{s,p}^{(1)}$.

\begin{lemma}(\cite{Yuan2024}*{Theorem 2.1}) \label{lem-noncrossing-characterization}
For $1 \le s \le p$, the set $\cC_{s,p}^{(1)}$ is exactly the set of all non-crossing sequences $\alpha \in \cC_{s,p}$.
\end{lemma}

We next present two results from \cite{Yuan2024} for non-crossing sequences that will be used in the moment calculation.
The first lemma indicates the graph $g(i,\alpha)$ for non-crossing $\alpha$.

\begin{lemma} {(\cite{Yuan2024}*{Proposition 2.1})} \label{lem-noncrossing-graph-dichotomy}
Let $\alpha\in\cC_{s,p}^{(1)}$ and let $i\in\cC_{r,p}$. Then either
\begin{align*}
    g(i,\alpha)\in\Delta^{(1)}(p;\alpha),
\end{align*}
or there exist $a\in[s]$ and $b\in[r]$ such that
\begin{align} \label{eq-unit-imbalance}
    |e_u(a,b)-e_d(a,b)|=1.
\end{align}
\end{lemma}

The second lemma provides the relationship between the unique leading graph and all the other graphs in $\Delta^{(1)}(p;\alpha)$ precisely.

\begin{lemma} {(\cite{Yuan2024}*{proof of Proposition 2.3})} \label{lem-gluing-bijection}
Fix $\alpha\in\cC_{s,p}^{(1)}$. Let $i_0\in\cC_{p+1-s,p}$ be the unique canonical sequence such that
\begin{align*}
    g(i_0,\alpha)\in\Delta^{(1)}_0(p;\alpha).
\end{align*}
Fix $1\le r\le p+1-s$. Consider a partition of $[p+1-s]$ into $r$ non-empty subsets $V_1,\ldots,V_r$. For each subset $V_t$ with $1 \le t \le r$, glue the $i$-vertices of $g(i_0,\alpha)$ whose values are in $V_t$, and then canonically relabel the resulting sequence. This construction defines a bijection from the set of such partitions to
\begin{align} \label{eq-Delta1-canonical-i}
    \{i\in\cC_{r,p}:g(i,\alpha)\in\Delta^{(1)}(p;\alpha)\}.
\end{align}
The set in \eqref{eq-Delta1-canonical-i} is empty when $p+1-s<r\le p$.
\end{lemma}

\section[Exact moment calculation]{Exact moment calculation}
\label{sec:moment}

We compute the moment
\begin{align*}
	\dfrac{1}{n^k} \bE \big[ \Tr M_{n,k,m}^p \big],
\end{align*}
for any $p \in \bN_+$.
Recalling the convention $\alpha_{p+1} = \alpha_{1}$ and the i.i.d. setting, we have
\begin{align} \label{eq-moment-1}
	\dfrac{1}{n^k} \bE \big[ \Tr M_{n,k,m}^p \big]
	=& \dfrac{1}{n^k} \sum_{\alpha_1, \ldots, \alpha_p = 1}^m \left( \prod_{t=1}^p \tau_{\alpha_t} \right) \bE \Big[ \Tr \big( Y_{\alpha_1} Y_{\alpha_1}^* Y_{\alpha_2} Y_{\alpha_2}^* \cdots Y_{\alpha_p} Y_{\alpha_p}^*\big) \Big] \nonumber \\
	=& \dfrac{1}{n^k} \sum_{\alpha_1, \ldots, \alpha_p = 1}^m \left( \prod_{t=1}^p \tau_{\alpha_t} \right) \bE \Bigg[ \prod_{l=1}^k \Tr \left( \left( \y_{\alpha_{1}}^{(l)} \right) \left( \y_{\alpha_{1}}^{(l)} \right)^* \ldots \y_{\alpha_p}^{(l)} \left( \y_{\alpha_p}^{(l)} \right)^* \right) \Bigg] \nonumber \\
	=& \dfrac{1}{n^k} \sum_{\alpha_1, \ldots, \alpha_p = 1}^m \left( \prod_{t=1}^p \tau_{\alpha_t} \right) \left( \bE \Bigg[ \Tr \left( \y_{\alpha_1}^{(1)} \left( \y_{\alpha_1}^{(1)} \right)^* \ldots \y_{\alpha_p}^{(1)} \left( \y_{\alpha_p}^{(1)} \right)^* \right) \Bigg] \right)^k \nonumber \\
    =& \dfrac{1}{n^k} \sum_{\alpha_1, \ldots, \alpha_p = 1}^m \left( \prod_{t=1}^p \tau_{\alpha_t} \right) S_n(\alpha)^k,
\end{align}
where
\begin{align} \label{eq-def-Sn}
    S_n(\alpha)
    :=& \bE\left[ \Tr\left( \y_{\alpha_1} \left( \y_{\alpha_1} \right)^* \cdots \y_{\alpha_p} \left( \y_{\alpha_p} \right)^* \right) \right],
\end{align}
for every $p\in\bN_+$ and every sequence $\alpha$.

For $i=(i^{(1)},\ldots,i^{(p)})\in[n]^p$, we write
\begin{align} \label{eq-def-E(i,alpha)}
    E(i,\alpha)
    :=& \bE\left[ \prod_{t=1}^p \overline{\left( \y_{\alpha_t} \right)_{i^{(t)}}} \left( \y_{\alpha_{t+1}} \right)_{i^{(t)}} \right].
\end{align}
For canonical sequences $\alpha \in \cC_{s,p}$ and $i \in \cC_{r,p}$ with $1\le r,s \le p$, we consider the corresponding graph $g(i,\alpha)$ given in Section \ref{sec:graph-def}.
We associate the random variable $\overline{\big( \y_{\alpha_t} \big)_{i^{(t)}}}$ with the down edge from $\alpha$-vertex $\alpha_t$ to $i$-vertex $i^{(t)}$ in $g(i,\alpha)$.
We also associate the random variable $\big( \y_{\alpha_{t+1}} \big)_{i^{(t)}}$ with the up edge from $i$-vertex $i^{(t)}$ to $\alpha$-vertex $\alpha_{t+1}$.
By the i.i.d. setting of the entries, we obtain from \eqref{eq-def-E(i,alpha)} that
\begin{align} \label{eq-E-factorization}
    E(i,\alpha)
    =& \prod_{a=1}^s\prod_{b=1}^r \bE \left[ \overline{\left( \y_a \right)_b}^{\,e_d(a,b)} \left( \y_a \right)_b^{e_u(a,b)} \right]
    = \prod_{a=1}^s \prod_{b=1}^r \bE \left[ \overline{\xi_1}^{e_d(a,b)} \xi_1^{e_u(a,b)} \right].
\end{align}

Recall the equivalence relationship defined in Section \ref{sec:graph-def}. By the i.i.d. setting, we have that $E(i,\alpha) = E(i',\alpha')$ if $g(i,\alpha) \sim g(i',\alpha')$.
Thus, expanding the trace in \eqref{eq-def-Sn} gives
\begin{align} \label{eq-Sn-canonical-sum}
    S_n(\alpha)
    = \sum_{i^{(1)},\ldots,i^{(p)}=1}^n E(i,\alpha)
    = \sum_{r=1}^{p} n(n-1)\cdots(n-r+1) \sum_{i\in\cC_{r,p}}E(i,\alpha).
\end{align}
Moreover, for any $\alpha\in\cC_{s,p}, \varphi\in\cI_{s,m}$, it holds that $S_n(\varphi(\alpha)) = S_n(\alpha)$.
Hence, for every $p\in\bN_+$, we obtain from \eqref{eq-moment-1} that
\begin{align} \label{eq-exact-raw-moment}
    \dfrac{1}{n^k} \bE\left[\Tr M_{n,k,m}^p\right]
    =& \dfrac{1}{n^k} \sum_{s=1}^{p}\sum_{\alpha \in \cC_{s,p}} \left( \sum_{\varphi\in\cI_{s,m}} \prod_{t=1}^p\tau_{\varphi(\alpha_t)} \right) S_n(\alpha)^k.
\end{align}

In order to compute \eqref{eq-exact-raw-moment}, we handle the two cases $\alpha \in \cC_{s,p}^{(1)}$ and $\alpha \in \cC_{s,p} \setminus \cC_{s,p}^{(1)}$ step by step.

\bigskip
\noindent \textbf{Step 1.} In this step, we focus on the case that $\alpha$ is non-crossing. We have the following lemma for $E(i,\alpha)$ given by \eqref{eq-def-E(i,alpha)}.

\begin{lemma} \label{lem-E-noncrossing}
Under Assumptions \ref{Assumption-distribution} and \ref{Assumption-mixed-moments}, let $\alpha\in\cC_{s,p}^{(1)}$ and let $i\in\cC_{r,p}$. Then
\begin{align} \label{eq-E-noncrossing}
    E(i,\alpha) =
    \begin{cases} \displaystyle
        \prod_{a=1}^s\prod_{b=1}^r m_{2e(a,b)}, & g(i,\alpha) \in \Delta^{(1)}(p;\alpha),\\
        0, & g(i,\alpha) \notin \Delta^{(1)}(p;\alpha).
    \end{cases}
\end{align}
\end{lemma}

\begin{proof}
If $g(i,\alpha) \in \Delta^{(1)}(p;\alpha)$, then $e_u(a,b) = e_d(a,b) = e(a,b)$ for every pair $(a,b)$, and the first case follows from \eqref{eq-E-factorization}.

Suppose that $g(i,\alpha) \notin \Delta^{(1)}(p;\alpha)$. By Lemma \ref{lem-noncrossing-graph-dichotomy}, there are $a\in[s]$ and $b\in[r]$ such that $|e_u(a,b)-e_d(a,b)|=1$. If both numbers are positive, the corresponding factor in \eqref{eq-E-factorization} is zero by Assumption \ref{Assumption-mixed-moments}. If one of them is zero, the other has to be one, and the corresponding factor is again zero because the entries of $\y$ are centered.
Hence, $E(i,\alpha)=0$.
\end{proof}

\begin{remark}
Under Assumptions \ref{Assumption-distribution} and \ref{Assumption-distribution-prime}, the conclusion of Lemma \ref{lem-E-noncrossing} follows more directly. In fact, it holds for every $\alpha \in \cC_{s,p}$.
If $g(i,\alpha) \notin \Delta^{(1)}(p;\alpha)$, then by the definition of $\Delta^{(1)}(p;\alpha)$, there are $a\in[s]$ and $b\in[r]$ such that $e_u(a,b) \not= e_d(a,b)$. The corresponding factor in \eqref{eq-E-factorization} is zero due to Assumption \ref{Assumption-distribution-prime}. Thus, Lemma \ref{lem-noncrossing-graph-dichotomy} and the separate treatment of a zero exponent are not necessary for this conclusion under Assumption \ref{Assumption-distribution-prime}.
\end{remark}

Under Assumptions \ref{Assumption-distribution} and \ref{Assumption-mixed-moments}, for $\alpha\in\cC_{s,p}^{(1)}$, Lemma \ref{lem-E-noncrossing} and the equivalence of graphs give
\begin{align*}
    S_n(\alpha)
    =\sum_{\substack{i\in[n]^p\\g(i,\alpha)\in\Delta^{(1)}(p;\alpha)}}E(i,\alpha).
\end{align*}

For $\alpha \in \cC_{s,p}^{(1)}$, Lemma \ref{lem-gluing-bijection} and \eqref{eq-Sn-canonical-sum} show that $S_n(\alpha)$ is exactly the weighted sum over all graphs obtained by gluing the $i$-vertices of the unique graph in $\Delta^{(1)}_0(p;\alpha)$. The following proposition first identifies a consecutive part of $\alpha$ and then uses a direct coordinate calculation to evaluate $S_n(\alpha)$ without enumerating the gluing separately.

\begin{proposition} \label{prop-Sn-recursion}
Let $\alpha=(\alpha_1,\ldots,\alpha_p) \in \cC_{s,p}^{(1)}$ with $2\le s\le p$.
\begin{enumerate}[(1)]
    \item There exist $a \in [s]$ and $1\le u\le v\le p$ such that
    \begin{align} \label{eq-alpha-consecutive-for-recursion}
        \{t\in[p]:\alpha_t=a\}=\{u,u+1,\ldots,v\}.
    \end{align}
    For such $a,u,v$, deleting $\alpha_u,\ldots,\alpha_v$ and canonically relabeling the remaining values gives a sequence $\alpha'$ in
    \begin{align*}
        \cC_{s-1,p-\deg_a(\alpha)}^{(1)}.
    \end{align*}
    
    \item Under Assumptions \ref{Assumption-distribution} and \ref{Assumption-mixed-moments}, let $a\in[s]$ and $1\le u\le v\le p$ be any choice satisfying \eqref{eq-alpha-consecutive-for-recursion}, and let $\alpha'$ be the canonical sequence obtained by deleting $\alpha_u,\ldots,\alpha_v$ and canonically relabeling the remaining values. Then
    \begin{align} \label{eq-Sn-recursion}
        S_n(\alpha)
        = \frac{ \bE \left[ \|\y\|^{2\deg_a(\alpha)} \right]}{n} S_n(\alpha').
    \end{align}
\end{enumerate}
\end{proposition}

\begin{proof}
We first prove part (1).
For each value appearing in $\alpha$, consider the difference between its largest and smallest positions of occurrence. Choose $a$ for which this difference is minimal, and set
\begin{align*}
    u=\min\{j\in[p]:\alpha_j=a\},
    \qquad
    v=\max\{j\in[p]:\alpha_j=a\}.
\end{align*}
If $u=v$, then \eqref{eq-alpha-consecutive-for-recursion} holds.

Suppose that $u<v$ and that $u<j<v$ for some $j$ with $\alpha_j \not= a$.
Write $b=\alpha_j$. Every occurrence of $b$ must lie strictly between $u$ and $v$. Indeed, an occurrence of $b$ at a position $j'<u$ would give
\begin{align*}
    j'<u<j<v,
    \qquad
    \alpha_{j'}=\alpha_j=b\not= a=\alpha_u=\alpha_v,
\end{align*}
whereas an occurrence of $b$ at a position $j'>v$ would give
\begin{align*}
    u<j<v<j',
    \qquad
    \alpha_u=\alpha_v=a\not= b=\alpha_j=\alpha_{j'}.
\end{align*}
Either case contradicts the non-crossing property. Hence, the difference between the largest and smallest positions at which $b$ occurs is strictly smaller than $v-u$, which contradicts the choice of $a$. This proves \eqref{eq-alpha-consecutive-for-recursion}.

We claim that deleting entries cannot create a crossing. Indeed, any four positions forming a crossing after the deletion already formed a crossing before the deletion.
Hence, the resulting sequence $\alpha'$ is also non-crossing, has $s-1$ distinct values, and has length $p-\deg_a(\alpha)$. The conclusion now follows from Lemma \ref{lem-noncrossing-characterization}.

We now prove part (2).
By the cyclic property of the trace, we obtain from \eqref{eq-def-Sn} that
\begin{align*}
    S_n((\alpha_1,\ldots,\alpha_p))
    =S_n((\alpha_{v+1},\ldots,\alpha_p,\alpha_1,\ldots,\alpha_v)).
\end{align*}
Note that cyclic shifts preserve the non-crossing property of a sequence. This shift places the chosen block \eqref{eq-alpha-consecutive-for-recursion} at the end. After the shift and canonical relabeling, we keep the same notation and assume that $2 \le u \le v = p$. The sequence left after deletion is a cyclic shift of $\alpha'$, up to canonical relabeling, so its value under $S_n$ is unchanged.

For a graph in $\Delta^{(1)}(p;\alpha)$, the down edges from the $\alpha$-vertex $a$ and the up edges to the $\alpha$-vertex $a$ are exactly
\begin{align} \label{eq-block-incident-edges}
    &a\longrightarrow i^{(t)},\qquad u\le t\le v, \nonumber\\
    &i^{(t)}\longrightarrow a,\qquad u-1\le t\le v-1.
\end{align}

For each $u\le t\le v-1$, the edges $a\to i^{(t)}$ and $i^{(t)}\to a$ give one edge in each direction. Subtracting these contributions from the equality $e_d(a,b)=e_u(a,b)$ for every $i$-vertex $b$ shows that the remaining edges $a\to i^{(v)}$ and $i^{(u-1)}\to a$ must have the same endpoint on the $i$-line. That is,
\begin{align} \label{eq-boundary-identification}
    i^{(v)}=i^{(u-1)}.
\end{align}
Note that \eqref{eq-boundary-identification} also holds whenever $u=v$.
We use the conventions that empty products equal one and sums over an empty tuple of indices have one term.
Subject to \eqref{eq-boundary-identification}, we delete the entries in positions $u,\ldots,v$ from both sequences $\alpha$ and $i$, together with the edges
\begin{align*}
    i^{(u-1)} \to \alpha_u, \alpha_u \to i^{(u)}, \ldots, \alpha_v \to i^{(v)}.
\end{align*}
The original edge $i^{(v)} \to \alpha_{v+1}$ is replaced by $i^{(u-1)} \to \alpha_{v+1}$. The remaining $\alpha$-vertices are unchanged, up to canonical relabeling. Deleting an occurrence in a sequence does not remove its graph vertex if its value also occurs at a retained position.
We provide an example of this deleting operation in Figure \ref{fig-deletion-consecutive-block} below.

\begin{figure}[htbp]
\centering
\begin{tikzpicture}[line width=0.65pt,>=latex,font=\small]
  \node[align=center] at (1.4,4.1)
    {$\alpha=(1,\underbrace{2,2}_{\text{delete}})$\\[6pt]
     $i=(1,\underbrace{2,1}_{\text{delete}})$};
  \draw[dotted] (-0.3,2) -- (3.1,2);
  \draw[dotted] (-0.3,0) -- (3.1,0);
  \node[circle,fill=black,inner sep=2pt] (a1) at (0,2) {};
  \node[circle,fill=black,inner sep=2pt] (a2) at (2.8,2) {};
  \node[circle,fill=black,inner sep=2pt] (b1) at (0,0) {};
  \node[circle,fill=black,inner sep=2pt] (b2) at (2.8,0) {};
  \node[above=4pt] at (a1) {$\alpha_1=1$};
  \node[above=4pt] at (a2) {$\alpha_2=\alpha_3=2$};
  \node[below=4pt] at (b1) {$i^{(1)}=i^{(3)}=1$};
  \node[below=4pt] at (b2) {$i^{(2)}=2$};
  \draw[->] (a1) to[bend left=13] (b1);
  \draw[->,dashed] (a2) to[bend left=13] (b2);
  \draw[->,dashed] (a2) to[bend left=10] (b1);
  \draw[->,dashed] (b1) to[bend left=10] (a2);
  \draw[->,dashed] (b2) to[bend left=13] (a2);
  \draw[->] (b1) to[bend left=13] (a1);
  \draw[->,line width=0.9pt] (4.1,1) -- (6.5,1);
  \node[align=center,above] at (5.3,1.13) {$a=2,\ u=2,\ v=3$\\delete positions $2,3$};
  \node[align=center,below] at (5.3,0.88) {$i^{(v)}=i^{(u-1)}=1$};
  \node[align=center] at (8.1,4.1) {$\alpha'=(1)$\\[6pt]$i'=(1)$};
  \draw[dotted] (7.1,2) -- (9.1,2);
  \draw[dotted] (7.1,0) -- (9.1,0);
  \node[circle,fill=black,inner sep=2pt] (c1) at (8.1,2) {};
  \node[circle,fill=black,inner sep=2pt] (d1) at (8.1,0) {};
  \node[above=4pt] at (c1) {$1$};
  \node[below=4pt] at (d1) {$1$};
  \draw[->] (c1) to[bend left=13] (d1);
  \draw[->] (d1) to[bend left=13] (c1);
\end{tikzpicture}
\caption{Deletion of a consecutive block with $a=2$, $u=2$, and $v=3$.}
\label{fig-deletion-consecutive-block}
\end{figure}
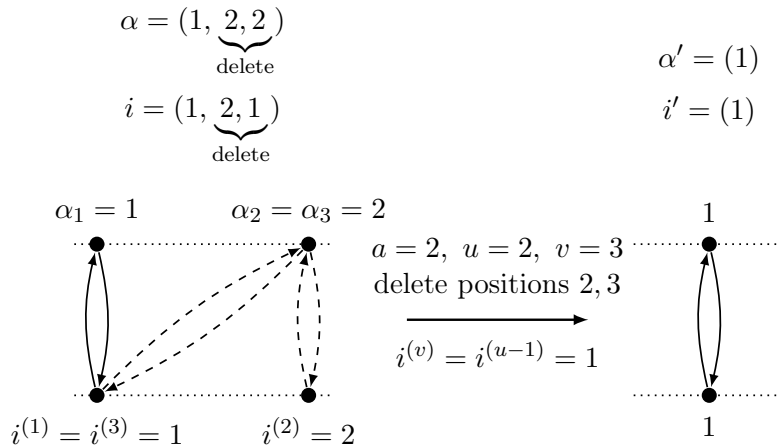

Conversely, start with an $i$ sequence for $\alpha'$ whose graph belongs to $\Delta^{(1)}(p-\deg_a(\alpha);\alpha')$. Choose $i^{(u)},\ldots,i^{(v-1)}$ arbitrarily in $[n]$, re-insert the deleted positions, and set $i^{(v)}=i^{(u-1)}$. The edges incident to the re-inserted $\alpha$-vertex $a$ are those in \eqref{eq-block-incident-edges}. Each down edge $a\to i^{(t)}$ and the corresponding up edge $i^{(t)}\to a$ have the same endpoints and opposite directions. The same holds for the remaining pair $a\to i^{(v)}$, $i^{(u-1)}\to a$. Hence $e_d(a,b)=e_u(a,b)$ for every $i$-vertex $b$. The edges incident to the remaining $\alpha$-vertices agree with those for $\alpha'$. This proves a bijection between the contributing index sequences for $\alpha$ and those for $\alpha'$, together with the unrestricted choices of the $i^{(u)},\ldots,i^{(v-1)}$.

Fix the indices outside the positions $u,\ldots,v$, including $i^{(u-1)}$.
For every $\deg_a(\alpha)\ge1$, the factors in \eqref{eq-def-E(i,alpha)} involving $\y_a$ are
\begin{align} \label{eq-block-y-factors}
    \left| \left( \y_a \right)_{i^{(u-1)}} \right|^2
    \prod_{t=u}^{v-1} \left| \left( \y_a \right)_{i^{(t)}} \right|^2.
\end{align}
Summing their expectation over the internal indices gives
\begin{align} \label{eq-internal-index-sum}
    \sum_{i^{(u)},\ldots,i^{(v-1)}=1}^n
    \bE\left[ \left| \left( \y_a \right)_{i^{(u-1)}} \right|^2 \prod_{t=u}^{v-1} \left| \left( \y_a \right)_{i^{(t)}} \right|^2 \right]
    = \bE\left[ \left| \left( \y_a \right)_{i^{(u-1)}}\right|^2 \left\| \y_a \right\|^{2(\deg_a(\alpha)-1)} \right].
\end{align}
As the entries of $\y_a$ are i.i.d., the last expectation in \eqref{eq-internal-index-sum} has the same value for every possible value of $i^{(u-1)}$, so averaging over this value gives
\begin{align*}
    &\sum_{i^{(u)},\ldots,i^{(v-1)}=1}^n
    \bE\left[ \left| \left( \y_a \right)_{i^{(u-1)}} \right|^2 \prod_{t=u}^{v-1} \left| \left( \y_a \right)_{i^{(t)}} \right|^2 \right] \\
    &= \dfrac{1}{n} \sum_{i^{(u-1)}=1}^n \bE\left[ \left| \left( \y_a \right)_{i^{(u-1)}}\right|^2 \left\| \y_a \right\|^{2(\deg_a(\alpha)-1)} \right]
    = \dfrac{1}{n} \bE \left[ \|\y\|^{2\deg_a(\alpha)} \right].
\end{align*}

By independence, equivalence, \eqref{eq-Sn-canonical-sum}, \eqref{eq-def-E(i,alpha)} and \eqref{eq-boundary-identification}, we obtain
\begin{align*}
    &S_n(\alpha)
    = \sum_{i^{(1)},\ldots,i^{(p)}=1}^n \bE\left[ \prod_{t=1}^p \overline{\left( \y_{\alpha_t} \right)_{i^{(t)}}} \left( \y_{\alpha_{t+1}} \right)_{i^{(t)}} \right] \\
    =& \sum_{\substack{i^{(1)},\ldots,i^{(u-1)},\\ i^{(v)},\ldots,i^{(p)}=1\\ i^{(v)}=i^{(u-1)}}}^n \Bigg( \bE\left[ \prod_{t\in\{1,\ldots,u-2\}\cup\{v+1,\ldots,p\}} \left( \overline{\left( \y_{\alpha_t} \right)_{i^{(t)}}} \left( \y_{\alpha_{t+1}} \right)_{i^{(t)}} \right) \overline{\left( \y_{\alpha_{u-1}} \right)_{i^{(u-1)}}} \left( \y_{\alpha_{v+1}} \right)_{i^{(v)}} \right] \\
    & \times \sum_{i^{(u)},\ldots,i^{(v-1)}=1}^n \bE\left[ \left| \left( \y_a \right)_{i^{(u-1)}} \right|^2 \prod_{t=u}^{v-1} \left| \left( \y_a \right)_{i^{(t)}} \right|^2 \right] \Bigg) \\
    =& \sum_{\substack{i^{(1)},\ldots,i^{(u-1)},\\ i^{(v+1)},\ldots,i^{(p)}=1}}^n \Bigg( \bE \left[ \prod_{t\in\{1,\ldots,u-2\}\cup\{v+1,\ldots,p\}} \left( \overline{\left( \y_{\alpha_t} \right)_{i^{(t)}}} \left( \y_{\alpha_{t+1}} \right)_{i^{(t)}} \right) \overline{\left( \y_{\alpha_{u-1}} \right)_{i^{(u-1)}}} \left( \y_{\alpha_{v+1}} \right)_{i^{(u-1)}} \right] \\
    & \times \sum_{i^{(u)},\ldots,i^{(v-1)}=1}^n \bE\left[ \left| \left( \y_a \right)_{i^{(u-1)}} \right|^2 \prod_{t=u}^{v-1} \left| \left( \y_a \right)_{i^{(t)}} \right|^2 \right] \Bigg) \\
    =& \dfrac{1}{n} S_n(\alpha') \bE \left[ \|\y\|^{2\deg_a(\alpha)} \right],
\end{align*}
This proves \eqref{eq-Sn-recursion}.
\end{proof}

\begin{corollary} \label{cor-Sn-exact}
Under Assumptions \ref{Assumption-distribution} and \ref{Assumption-mixed-moments}, for every $\alpha\in\cC_{s,p}^{(1)}$,
\begin{align} \label{eq-Sn-exact}
    S_n(\alpha)
    = n^{1-s} \prod_{a=1}^s \bE \left[ \|\y\|^{2\deg_a(\alpha)} \right].
\end{align}
\end{corollary}

\begin{proof}
We prove the identity by induction on the number $s$ of distinct values in $\alpha$, simultaneously for all sequence lengths.

If $s=1$, then $\alpha=(1,\ldots,1)$ so $\deg_1(\alpha) = p$. We have
\begin{align*}
    S_n(\alpha)
    = \bE \left[ \Tr \left( \left[ \y_1 \left( \y_1 \right)^* \right]^p \right) \right]
    = \bE \left[ \|\y\|^{2p} \right].
\end{align*}
This is \eqref{eq-Sn-exact} for $s=1$.

Now let $s\ge2$ and assume that \eqref{eq-Sn-exact} holds for every non-crossing canonical sequence with $s-1$ distinct values, regardless of its length. By part (1) of Proposition \ref{prop-Sn-recursion}, choose $a,u,v$ satisfying \eqref{eq-alpha-consecutive-for-recursion}, and let $\alpha'$ be the sequence obtained after deletion and canonical relabeling. Then
\begin{align*}
    \alpha' \in \cC_{s-1,p-\deg_a(\alpha)}^{(1)}.
\end{align*}
Part (2) of Proposition \ref{prop-Sn-recursion} and the induction hypothesis give
\begin{align*}
    S_n(\alpha)
    = \frac{ \bE \left[ \|\y\|^{2\deg_a(\alpha)} \right]}{n} S_n(\alpha')
    = \frac{ \bE \left[ \|\y\|^{2\deg_a(\alpha)} \right]}{n} n^{1-(s-1)}
    \prod_{b=1}^{s-1} \bE \left[ \|\y\|^{2\deg_b(\alpha')} \right].
\end{align*}
By the definition of $\alpha'$,
\begin{align*}
    \prod_{b=1}^{s-1} \bE \left[ \|\y\|^{2\deg_b(\alpha')} \right]
    = \prod_{\substack{b=1\\b\ne a}}^s \bE \left[ \|\y\|^{2\deg_b(\alpha)} \right].
\end{align*}
Consequently,
\begin{align*}
    S_n(\alpha)
    = n^{1-s}\prod_{b=1}^s \bE \left[ \|\y\|^{2\deg_b(\alpha)} \right].
\end{align*}
The proof is concluded.
\end{proof}

Next, we verify that the norm moments appearing in \eqref{eq-Sn-exact} are exactly the quantities $H_n(d)$ defined in \eqref{eq-H-explicit}, and establish the recurrence \eqref{eq-H-degree-recursion} for $H_n(d)$.

\begin{lemma} \label{lem-norm-moment-polynomial}
Under Assumption \ref{Assumption-distribution}, for integers $n\ge1$ and $d\ge0$,
\begin{align*}
    H_n(d)= \bE \left[ \|\y\|^{2d} \right].
\end{align*}
For every fixed $d\in\bN_+$, $H_n(d)$ is a polynomial in $n$ of degree $d$ with leading coefficient one.
Moreover, the recurrence \eqref{eq-H-degree-recursion} holds with the initial values stated in Remark \ref{rmk-H-degree-recursion}.
\end{lemma}

\begin{proof}
The case $d=0$ is trivial and we only consider $d \ge 1$ below. Recall that $\y=(\xi_1,\ldots,\xi_n)$. For $n,d \ge 1$,
\begin{align*}
    \bE \left[ \|\y\|^{2d} \right]
    = \bE \left[ \left( |\xi_1|^2 + \ldots + |\xi_n|^2 \right)^d \right]
    = \sum_{b_1,\ldots,b_d=1}^n \bE\left[ \prod_{j=1}^d |\xi_{b_j}|^2 \right].
\end{align*}
We group the terms according to the equality relation among $b_1,\ldots,b_d$.
For $\pi \in \cP(d)$, the indices agree within each block and are distinct across different blocks. There are $(n)_{|\pi|}$ ways to assign distinct indices to the blocks. By independence, each assignment has an expectation of $\prod_{B\in\pi}m_{2|B|}$. Consequently, by \eqref{eq-H-explicit},
\begin{align*}
    \bE \left[ \|\y\|^{2d} \right]
    = \sum_{\pi\in\cP(d)} (n)_{|\pi|} \prod_{B\in\pi} m_{2|B|}
    = H_n(d).
\end{align*}

For fixed $d \ge 1$, each term of \eqref{eq-H-explicit} is a polynomial in $n$ of degree $|\pi|$. The partition into $d$ singletons is the unique partition with $|\pi|=d$. Its contribution is $(n)_d$ because $m_2=1$, whereas every other partition has at most $d-1$ blocks.
Hence, the degree is $d$ and the leading coefficient is one.

To establish the recurrence, consider the formal power series
\begin{align*}
    A(z) = \sum_{\ell=0}^{\infty} m_{2\ell} \frac{z^\ell}{\ell!},
\end{align*}
where $z$ is a formal variable.
For $n \ge 1$ and $j \ge 0$, we have
\begin{align*}
    H_n(j) = \bE \left[ \left( \sum_{b=1}^n |\xi_b|^2 \right)^j \right]
    = \sum_{\substack{\ell_1, \ldots, \ell_n \ge 0 \\ \ell_1 + \cdots + \ell_n = j}} \frac{j!}{\ell_1! \cdots \ell_n!} \prod_{b=1}^n m_{2\ell_b}.
\end{align*}
Thus, we obtain
\begin{align*}
    A(z)^n = \sum_{j=0}^{\infty} \left( \sum_{\substack{\ell_1, \ldots, \ell_n \ge 0 \\ \ell_1 + \cdots + \ell_n = j}} \prod_{b=1}^n \frac{m_{2\ell_b}}{\ell_b!} \right) z^j
    = \sum_{j=0}^{\infty} H_n(j) \frac{z^j}{j!},
    \qquad \forall n \ge 1.
\end{align*}

Formal differentiation gives
\begin{align*}
    A(z)\frac{d}{dz}\bigl(A(z)^n\bigr)
    =nA'(z)A(z)^n.
\end{align*}
For $d\ge1$, comparing coefficients of $z^{d-1}$ and multiplying by $(d-1)!$ yields
\begin{align*}
    \sum_{\ell=0}^{d-1}\binom{d-1}{\ell}
    m_{2\ell}H_n(d-\ell)
    =n\sum_{\ell=1}^{d}\binom{d-1}{\ell-1}
    m_{2\ell}H_n(d-\ell).
\end{align*}
Isolating the $\ell=0$ term on the left and using $m_0=1$ and $\binom{d-1}{d}=0$ gives \eqref{eq-H-degree-recursion}.
\end{proof}

\bigskip
\noindent \textbf{Step 2.} In this step, we focus on the case that $\alpha$ is crossing.
For crossing sequences under Assumptions \ref{Assumption-distribution} and \ref{Assumption-mixed-moments}, we compute $E(i,\alpha)$.

We first start with the following lemma, which provides a bound on the number of vertices of $g(i,\alpha)$, and which can be found in \cite{Yuan2024}.
\begin{lemma}(\cite{Yuan2024}*{Proposition 2.2}) \label{lem-yuan-vertex-bound}
For $1 \le r,s \le p$, let $\alpha \in \cC_{s,p}$ and $i \in \cC_{r,p}$.
\begin{enumerate}[(1)]
    \item If $g(i,\alpha) \in \Delta^{(1)}(p;\alpha)$, then $r+s \le p+1$. Equality holds if and only if $g(i,\alpha) \in \Delta^{(1)}_0(p;\alpha)$.
    \item If $g(i,\alpha) \in \Delta^{(2)}(p;\alpha)$ and
    \begin{align*}
        |e_u(a,b)-e_d(a,b)| \not= 1 \qquad \text{for every } a \in [s],b\in[r],
    \end{align*}
    then $r+s \le p$.
\end{enumerate}
\end{lemma}

\begin{corollary} \label{lem-nonzero-vertex-bound}
Under Assumptions \ref{Assumption-distribution} and \ref{Assumption-mixed-moments}, let $\alpha \in \cC_{s,p}$ and $i \in \cC_{r,p}$. If $E(i,\alpha) \not= 0$, then
\begin{align} \label{eq-general-vertex-bound}
    r+s\le p+1.
\end{align}
If equality holds, then $g(i,\alpha) \in \Delta^{(1)}_0(p;\alpha)$.
Consequently, if $\alpha$ is crossing and $E(i,\alpha)\ne0$, then
\begin{align} \label{eq-crossing-vertex-bound}
    r+s\le p.
\end{align}
\end{corollary}

\begin{proof}
If $|e_u(a,b)-e_d(a,b)|=1$ for some pair $(a,b)$, then we have $\bE[\overline{\xi_1}^{e_d(a,b)} \xi_1^{e_u(a,b)}] = 0$, which implies that $E(i,\alpha)=0$ by \eqref{eq-E-factorization}. Indeed, when both $e_u(a,b)$ and $e_d(a,b)$ are positive, this follows from Assumption \ref{Assumption-mixed-moments}. When one of $e_u(a,b)$ and $e_d(a,b)$ is zero, the other is one. As $\xi_1$ is centered, we have the same conclusion.
Thus, $E(i,\alpha) \not= 0$ implies that no such pair exists.

The two parts of Lemma \ref{lem-yuan-vertex-bound} give \eqref{eq-general-vertex-bound} and the condition under which the equality holds.
If $\alpha$ is crossing, Lemma \ref{lem-noncrossing-characterization} excludes a graph in $\Delta^{(1)}_0(p;\alpha)$, so equality cannot hold. This proves \eqref{eq-crossing-vertex-bound}.
\end{proof}

\begin{remark}
Under Assumptions \ref{Assumption-distribution} and \ref{Assumption-distribution-prime}, \eqref{eq-E-factorization} gives $E(i,\alpha) \not= 0$ only if $e_u(a,b)=e_d(a,b)$ for every pair $(a,b)$. Hence $g(i,\alpha) \in \Delta^{(1)}(p;\alpha)$, and only part (1) of Lemma \ref{lem-yuan-vertex-bound} is needed in the proof of Corollary \ref{lem-nonzero-vertex-bound}. The crossing case follows in the same way.
\end{remark}

\begin{corollary} \label{cor-crossing-Sn-bound}
Under Assumptions \ref{Assumption-distribution} and \ref{Assumption-mixed-moments}, for fixed $p$, there is a finite constant $C_p$, which is independent of $n,k,m$, and $\alpha$, and which depends only on $p$ and $m_1,\ldots,m_{2p}$, such that, for every crossing $\alpha \in \cC_{s,p}$,
\begin{align} \label{eq-crossing-Sn-bound}
    |S_n(\alpha)|\le C_p n^{p-s}.
\end{align}
\end{corollary}

\begin{proof}
We consider the equation \eqref{eq-Sn-canonical-sum}. Corollary \ref{lem-nonzero-vertex-bound} shows that a nonzero term must have $r\le p-s$.
For every pair $(a,b)$, we have
\begin{align*}
    \left| \bE \left[ \overline{\left( \y_a \right)_b}^{\,e_d(a,b)} \left( \y_a \right)_b^{e_u(a,b)} \right] \right|
    \le m_{e_d(a,b)+e_u(a,b)}.
\end{align*}
The total number of edges in the graph $g(i,\alpha)$ is $2p$. There are at most $2p$ pairs for which $e_d(a,b)+e_u(a,b)>0$. For such pairs, it holds that $e_d(a,b)+e_u(a,b) \le 2p$.
Thus, by \eqref{eq-E-factorization}, we can bound $|E(i,\alpha)|$ by a constant $C_p$ depending only on $p$ and $m_1,\ldots,m_{2p}$.
For fixed $p$, the number of canonical sequences $i$ is finite, independently of $n$. Finally,
\begin{align*}
    n(n-1)\cdots(n-r+1) \le n^r \le n^{p-s}.
\end{align*}
Combining \eqref{eq-E-factorization}, \eqref{eq-Sn-canonical-sum} and these observations gives \eqref{eq-crossing-Sn-bound}.
\end{proof}

\bigskip

Now we are able to compute the moment of the matrix $M_{n,k,m}$ from \eqref{eq-exact-raw-moment} by combining the non-crossing result \eqref{eq-Sn-exact}, Lemma \ref{lem-norm-moment-polynomial}, and the crossing bound \eqref{eq-crossing-Sn-bound}.

\begin{theorem} \label{thm-normalized-moment}
Under Assumptions \ref{Assumption-distribution} and \ref{Assumption-mixed-moments}, for every fixed $p\in\bN_+$,
\begin{align*} 
    \dfrac{1}{n^k} \bE \left[ \Tr M_{n,k,m}^p \right] =& \dfrac{1}{n^k} \sum_{s=1}^{p} \sum_{\alpha \in \cC_{s,p}^{(1)}} \left( \sum_{\varphi \in \cI_{s,m}} \prod_{t=1}^p \tau_{\varphi(\alpha_t)} \right) \left[ n^{1-s} \prod_{a=1}^s H_n(\deg_a(\alpha)) \right]^k + R_{n,k,m}^{(p)},
\end{align*}
where
\begin{align*}
    R_{n,k,m}^{(p)}
    :=& \dfrac{1}{n^k} \sum_{s=1}^{p} \sum_{\alpha \in \cC_{s,p} \setminus \cC_{s,p}^{(1)}} \left( \sum_{\varphi \in \cI_{s,m}} \prod_{t=1}^p \tau_{\varphi(\alpha_t)} \right) S_n(\alpha)^k
\end{align*}
satisfies the following bound
\begin{align*}
    |R_{n,k,m}^{(p)}|
    \le& \sum_{s=1}^{p} \sum_{\alpha \in \cC_{s,p} \setminus \cC_{s,p}^{(1)}} \left| \sum_{\varphi \in \cI_{s,m}} \prod_{t=1}^p \tau_{\varphi(\alpha_t)} \right| \left( C_p n^{p-s-1} \right)^k.
\end{align*}
\end{theorem}

\begin{proof}
Split the sum of $\alpha$ in \eqref{eq-exact-raw-moment} into non-crossing and crossing sequences. Then apply Corollary \ref{cor-Sn-exact} and Lemma \ref{lem-norm-moment-polynomial} to the non-crossing part and Corollary \ref{cor-crossing-Sn-bound} to the crossing part to conclude the proof.
\end{proof}

\section[Asymptotic moment calculation]{Asymptotic moment calculation}
\label{sec:asymptotic-moment}

In this section, we develop the proof of Theorem \ref{Thm-main} with the help of Theorem \ref{thm-normalized-moment}.

We first establish the following quantitative estimate.
\begin{theorem} \label{thm-moment-relative-error}
Under the assumptions of Theorem \ref{Thm-main}, for every fixed $p\in\bN_+$,
\begin{align} \label{eq-main-relative-error}
    \frac1{n^k} \bE \left[ \Tr M_{n,k,m}^p\right]
    =& \Bigg[ \sum_{s=1}^p c^s \sum_{\alpha\in\cC_{s,p}^{(1)}}
    \left(\prod_{a=1}^s m_{\deg_a(\alpha)}^{(\tau)} \right)
    \left(\prod_{a=1}^s H_n(\deg_a(\alpha))\right)^k
    \Bigg] \nonumber\\
    &\quad\times
    \Bigg[1+O\left(
    \left|\frac{m}{n^k}-c\right|
    +\max_{1\le q\le p}
    \left|\frac1m\sum_{j=1}^m\tau_j^q-m_q^{(\tau)}\right|
    +\left(\frac{C_p}{n}\right)^k
    \right)\Bigg].
\end{align}
The constant $C_p\ge1$ may be chosen as in Corollary \ref{cor-crossing-Sn-bound}, after increasing it if necessary.
The implicit constant in \eqref{eq-main-relative-error} may depend on $p$, $c$, the distribution of $\xi_1$, and $m_1^{(\tau)},\ldots,m_p^{(\tau)}$, but not on $n,k,m$.
\end{theorem}

\begin{proof}[Proof of Theorem \ref{thm-moment-relative-error}]
By Theorem \ref{thm-normalized-moment}, we have
\begin{align} \label{eq-asymptotic-exact-start}
    &\dfrac{1}{n^k} \bE\left[\Tr M_{n,k,m}^p\right] \nonumber \\
    =& \dfrac{1}{n^k} \sum_{s=1}^{p} m^s \sum_{\alpha \in \cC_{s,p}^{(1)}} \left( \dfrac{1}{m^s} \sum_{\varphi \in \cI_{s,m}} \prod_{t=1}^p \tau_{\varphi(\alpha_t)} \right) \left[ n^{1-s} \prod_{a=1}^s \bE \left[ \|\y\|^{2\deg_a(\alpha)} \right] \right]^k +R_{n,k,m}^{(p)} \nonumber \\
    =& \dfrac{1}{n^k} \sum_{s=1}^{p} m^s \sum_{\alpha \in \cC_{s,p}^{(1)}}\left( \frac1{m^s}\sum_{\varphi\in\cI_{s,m}}     \prod_{a=1}^s\tau_{\varphi(a)}^{\deg_a(\alpha)} \right) \left[ n^{1-s} \prod_{a=1}^s \bE \left[ \|\y\|^{2\deg_a(\alpha)} \right] \right]^k +R_{n,k,m}^{(p)} \nonumber \\
    =& \sum_{s=1}^p\left(\frac{m}{n^k}\right)^s
    \sum_{\alpha\in\cC_{s,p}^{(1)}}
    \left( \sum_{\varphi \in \cI_{s,m}} \prod_{a=1}^s \dfrac{1}{m} \tau_{\varphi(a)}^{\deg_a(\alpha)} \right)
    \left(\prod_{a=1}^s \bE \left[ \|\y\|^{2\deg_a(\alpha)} \right] \right)^k
    +R_{n,k,m}^{(p)}.
\end{align}

We estimate the weight coefficient. For every fixed $1\le q\le p$, \eqref{eq-tau-moment-convergence} gives
\begin{align} \label{eq-O(1) bound-tau}
    \frac1m\sum_{j=1}^m\tau_j^q=O(1).
\end{align}
For $1\le s\le p$, for $\alpha\in\cC_{s,p}$, the product $\prod_{a=1}^s \left( \frac1m \sum_{j=1}^m \tau_j^{\deg_a(\alpha)} \right)$ sums over all maps from $[s]$ to $[m]$, while the coefficient in \eqref{eq-asymptotic-exact-start} sums only over injective maps.
Every non-injective map identifies at least one pair $a<b$. Since the weights are nonnegative,
\begin{align} \label{eq-tau-collision-bound}
    0
    \le& \prod_{a=1}^s \left( \frac1m \sum_{j=1}^m \tau_j^{\deg_a(\alpha)} \right)
    - \sum_{\varphi \in \cI_{s,m}} \prod_{a=1}^s \dfrac{1}{m} \tau_{\varphi(a)}^{\deg_a(\alpha)} \nonumber \\
    \le& \frac1m \sum_{1\le a<b\le s} \left( \frac1m \sum_{j=1}^m \tau_j^{\deg_a(\alpha)+\deg_b(\alpha)} \right) \prod_{\substack{1\le t\le s\\t\ne a,b}}
    \left( \frac1m \sum_{j=1}^m \tau_j^{\deg_t(\alpha)} \right)
    =O(m^{-1}).
\end{align}

We start with the non-crossing sequence $\alpha$ and compute each product factor one by one in the following.
Firstly, by \eqref{eq-tau-moment-convergence} and \eqref{eq-tau-collision-bound}, we obtain
\begin{align} \label{eq-tau-coefficient-error}
    \frac1{m^s} \sum_{\varphi\in\cI_{s,m}} \prod_{a=1}^s \tau_{\varphi(a)}^{\deg_a(\alpha)}
    = \prod_{a=1}^s m_{\deg_a(\alpha)}^{(\tau)}
    +O\left( \max_{1\le q\le p} \left| \frac1m \sum_{j=1}^m \tau_j^q - m_q^{(\tau)} \right| +\frac1m \right).
\end{align}
Secondly, Assumption \ref{Assumption-ratio} implies
\begin{align*}
    \left|\left(\frac{m}{n^k}\right)^s-c^s\right|
    =O\left(\left|\frac{m}{n^k}-c\right|\right),
    \qquad 1\le s\le p.
\end{align*}
Thirdly, since $\sum_{a=1}^s\deg_a(\alpha)=p$, H\"older's inequality gives
\begin{align} \label{eq-norm-product-bound}
    \prod_{a=1}^s \bE \left[ \|\y\|^{2\deg_a(\alpha)} \right]
    \le \prod_{a=1}^s \left( \bE \left[ \|\y\|^{2p} \right] \right)^{\deg_a(\alpha)/p}
    \le \bE \left[ \|\y\|^{2p} \right].
\end{align}
Consequently, we have
\begin{align} \label{eq-noncrossing-coefficient-error}
    &\bigg| \sum_{s=1}^p\left(\frac{m}{n^k}\right)^s
    \sum_{\alpha\in\cC_{s,p}^{(1)}}
    \left(\frac1{m^s}\sum_{\varphi\in\cI_{s,m}}
    \prod_{a=1}^s\tau_{\varphi(a)}^{\deg_a(\alpha)}\right)
    \left(\prod_{a=1}^s \bE \left[ \|\y\|^{2\deg_a(\alpha)} \right] \right)^k \nonumber \\
    &\qquad\qquad \qquad\qquad \qquad\qquad - \sum_{s=1}^p c^s \sum_{\alpha\in\cC_{s,p}^{(1)}}
    \left(\prod_{a=1}^s m_{\deg_a(\alpha)}^{(\tau)} \right)
    \left(\prod_{a=1}^s H_n(\deg_a(\alpha))\right)^k \bigg| \nonumber \\
    \le& \left( \bE \left[ \|\y\|^{2p} \right] \right)^k
    O\left( \left| \frac{m}{n^k} - c \right| +\max_{1\le q\le p} \left| \frac1m \sum_{j=1}^m \tau_j^q - m_q^{(\tau)} \right| +\frac1m \right).
\end{align}

We next estimate the contribution of the crossing sequence $\alpha$. For every $\alpha\in\cC_{s,p}$, by \eqref{eq-O(1) bound-tau} and \eqref{eq-tau-collision-bound} imply
\begin{align*}
    0 \le
    \sum_{\varphi\in\cI_{s,m}} \prod_{a=1}^s \frac1{m} \tau_{\varphi(a)}^{\deg_a(\alpha)}
    \le \prod_{a=1}^s \left( \frac1m \sum_{j=1}^m \tau_j^{\deg_a(\alpha)} \right)
    =O(1).
\end{align*}
The remainder estimate in Theorem \ref{thm-normalized-moment}, the boundedness of $m/n^k$, and the finiteness of the sets of canonical sequences of length $p$ therefore give
\begin{align*}
    |R_{n,k,m}^{(p)}|
    =O\left(C_p^k n^{k(p-1)}\right).
\end{align*}
Since $\bE \left[ \|\y\|^2 \right] =n$, Jensen's inequality gives $\bE \left[ \|\y\|^{2p} \right] \ge n^p$. It follows that
\begin{align} \label{eq-crossing-relative-scale}
    \frac{|R_{n,k,m}^{(p)}|}{\left( \bE \left[ \|\y\|^{2p} \right] \right)^k}
    =O\left(\left(\frac{C_p}{n}\right)^k\right).
\end{align}

The non-crossing sequence $\alpha=(1,\ldots,1)\in\cC_{1,p}^{(1)}$ contributes exactly
\begin{align*}
    c m_p^{(\tau)} \left( \bE \left[ \|\y\|^{2p} \right] \right)^k.
\end{align*}
The Jensen's inequality implies $m_p^{(\tau)} \ge \left(m_1^{(\tau)} \right)^p > 0$.
Thus, by the nonnegativity, the sum for non-crossing $\alpha$ in \eqref{eq-main-relative-error} is bounded below by $c m_p^{(\tau)} \left( \bE \left[ \|\y\|^{2p} \right] \right)^k$.
Assumption \ref{Assumption-ratio} implies that $m^{-1} = O(n^{-k})$. Combining this observation and the lower bound with \eqref{eq-noncrossing-coefficient-error} and \eqref{eq-crossing-relative-scale} completes the proof.
\end{proof}

\begin{proof}[Proof of Theorem \ref{Thm-main}]
With the help of Assumption \ref{Assumption-ratio} and \eqref{eq-tau-moment-convergence}, Theorem \ref{Thm-main} follows from Theorem \ref{thm-moment-relative-error}.
\end{proof}

\section{Variance Calculation}
\label{sec:variance}

In this section, we prove Theorem \ref{Thm-variance} following the variance calculation in \cite{BYY2022}*{Section 3.2}. We recall the vectors $\y_\alpha^{(l)}$ and matrix $M_{n,k,m}$ introduced in Section \ref{sec:intro}.

\begin{proof}[Proof of Theorem \ref{Thm-variance}]
Fix $p \in \bN_+$. Throughout the proof, $C_p$ denotes a finite positive constant independent of $n,k,m$, which may vary from line to line.
For $\alpha,\beta \in [m]^p$, we use the conventions $\alpha_{p+1} = \alpha_1$ and $\beta_{p+1} = \beta_1$. We write $\alpha \cap \beta = \varnothing$ if the two sequences have no common value. Recall $E(i,\alpha)$ and $S_n(\alpha)$ from \eqref{eq-def-E(i,alpha)} and \eqref{eq-def-Sn}, respectively. For $i,j \in [n]^p$, define
\begin{align} \label{eq-def-E'}
    E'(i,\alpha;j,\beta)
    := \bE \left[ \prod_{t=1}^p
    \overline{\left( \y_{\alpha_t} \right)_{i^{(t)}}}
    \left( \y_{\alpha_{t+1}} \right)_{i^{(t)}}
    \overline{\left( \y_{\beta_t} \right)_{j^{(t)}}}
    \left( \y_{\beta_{t+1}} \right)_{j^{(t)}} \right].
\end{align}
Since $M_{n,k,m}$ is Hermitian, $\Tr M_{n,k,m}^p$ is real. Expanding the trace and using the independence of the tensor factors as in \eqref{eq-moment-1}, we obtain
\begin{align} \label{eq-variance-expansion}
    &\Var \left( \frac1{n^k} \Tr M_{n,k,m}^p \right) \nonumber \\
    &= \frac1{n^{2k}} \sum_{\alpha,\beta \in [m]^p}
    \left( \prod_{t=1}^p \tau_{\alpha_t} \tau_{\beta_t} \right)
    \left\{
    \left( \sum_{i,j \in [n]^p} E'(i,\alpha;j,\beta) \right)^k
    - \left( S_n(\alpha) S_n(\beta) \right)^k
    \right\}.
\end{align}
When $\alpha \cap \beta = \varnothing$, independence gives
\begin{align*}
    E'(i,\alpha;j,\beta)
    &= E(i,\alpha) E(j,\beta),
\end{align*}
so the corresponding summand in \eqref{eq-variance-expansion} is zero. It remains to consider $\alpha \cap \beta \ne \varnothing$.

Fix a pair $(\alpha,\beta)$ of sequences satisfying $\alpha \cap \beta \ne \varnothing$, and let $s$ be the number of distinct values in the sequence $(\alpha,\beta) = (\alpha_1,\ldots,\alpha_p,\beta_1,\ldots,\beta_p)$.
For $i,j \in [n]^p$, let $r$ be the number of distinct values in the sequence $(i,j)$.
We join the graphs $g(i,\alpha)$ and $g(j,\beta)$ in the following way.
We identify the $\alpha$-vertices $\alpha_1,\ldots,\alpha_p,\beta_1,\ldots,\beta_p$ on the same $\alpha$-line with the same value, and identify the $i$-vertices $i_1,\ldots,i_p,j_1,\ldots,j_p$ on the same $i$-line with the same value.
We retain all coincident edges.
The resulting graph is connected, since $\alpha$ and $\beta$ have a common value, and it has $4p$ edges counted with multiplicity.

If there are exactly one edge associated with a pair of $\alpha$-vertex $a$ and $i$-vertex $b$, then the corresponding entry $(\y_a)_b$ occurs only once in the product in \eqref{eq-def-E'}.
Independence and centering therefore imply that $E'(i,\alpha;j,\beta)=0$.
This particular edge must belong to only one of the two graphs $g(i,\alpha)$ and $g(j,\beta)$. Thus, the entry $(\y_a)_b$ occurs only once in the two expectations $E(i,\alpha)$ and $E(j,\beta)$, which means that $E(i,\alpha) E(j,\beta) = 0$.
Consequently, whenever either expression is nonzero, every pair of adjacent vertices is joined by at least two edges.
There are $4p$ edges counted with multiplicity, so at most $2p$ pairs of vertices are adjacent.
Since the graph is connected, it follows that
\begin{align} \label{eq-variance-vertex-bound}
    r+s-1 \le 2p.
\end{align}
By the H\"older's inequality, we have
\begin{align*}
    |E'(i,\alpha;j,\beta)|
    &\le m_{4p},
    \qquad |E(i,\alpha) E(j,\beta)|
    \le m_{2p}^2 \le m_{4p}.
\end{align*}
For fixed $p$, there are finitely many canonical sequences $(i,j)$ of length $2p$. Each canonical sequence with $r$ distinct values corresponds to $(n)_r \le n^r$ equivalent sequences in $[n]^{2p}$.
Thus, using $r \le 2p+1-s$ from \eqref{eq-variance-vertex-bound} and summing over the coordinate indices, we obtain
\begin{align} \label{eq-variance-single-factor-bound}
    \left| \sum_{i,j \in [n]^p} E'(i,\alpha;j,\beta) \right|
    + |S_n(\alpha) S_n(\beta)|
    &\le C_p n^{2p+1-s}.
\end{align}

We next sum over the column indices. We sum over canonical sequences $(\alpha,\beta)$ of length $2p$ and use the injective maps in $\cI_{s,m}$ to obtain all equivalent sequences in $[m]^{2p}$. We set $\deg_a(\alpha)=0$ when $a$ does not occur in $\alpha$, and similarly for $\beta$. Since the weights are nonnegative, \eqref{eq-tau-moment-convergence} gives
\begin{align} \label{eq-variance-weight-bound}
    &\frac1{m^s} \sum_{\varphi \in \cI_{s,m}}
    \prod_{a=1}^s \tau_{\varphi(a)}^{\deg_a(\alpha)+\deg_a(\beta)}
    \le \prod_{a=1}^s \left( \frac1m \sum_{j=1}^m
    \tau_j^{\deg_a(\alpha)+\deg_a(\beta)} \right)
    = O(1).
\end{align}
Substituting \eqref{eq-variance-single-factor-bound} and \eqref{eq-variance-weight-bound} to \eqref{eq-variance-expansion}, and grouping the column indices by their canonical sequences, we obtain
\begin{align*}
    \Var \left( \frac1{n^k} \Tr M_{n,k,m}^p \right)
    \le C_p^k n^{(2p-1)k}
    \sum_{s=1}^{2p-1} \left( \frac{m}{n^k} \right)^s
    \sum_{\substack{(\alpha,\beta) \in \cC_{s,2p} \\
    \alpha \cap \beta \ne \varnothing}}
    \left( \prod_{a=1}^s \left( \frac1m \sum_{j=1}^m
    \tau_j^{\deg_a(\alpha)+\deg_a(\beta)} \right) \right).
\end{align*}
The proof of \eqref{eq-main-variance} is concluded by Assumption \ref{Assumption-ratio} and \eqref{eq-tau-moment-convergence}, and increasing $C_p$ if necessary.
\end{proof}

\section{Applications to the spectrum}
\label{sec:examples}

In this section, we apply Theorems \ref{Thm-main} and \ref{Thm-variance} to four cases.
In Section \ref{sec:example-unit-modulus}, we treat $m_4=1$ and recover part of \cite{Yuan2024}.
In Section \ref{sec:example-small-k}, we establish the limiting moment sequence in the regime of \cite{Lytova2018}*{Theorem 1.2 and Remark 4.1}.
In Section \ref{sec:example-linear-k}, we recover the limiting moment formula in \cite{BYY2022}*{Theorem 2.1}.
In Section \ref{sec:example-quadratic-k}, under the conditions $m_4>1$ and $k/n \to \infty$, we compute the growth of the $p$th moment for $\widetilde M_{n,k,m}$ for all fixed $p \in \bN_+$, both in expectation and almost surely. We also develop a lower bound for the largest eigenvalue in Section \ref{sec:example-quadratic-k}.

Throughout this section, the assumptions of Theorem \ref{Thm-main} hold.
We focus on the normalized matrix model $\widetilde M_{n,k,m} = n^{-k} M_{n,k,m}$ introduced in Section \ref{sec:intro}.
We have the following corollary, which is a combination of Theorems \ref{Thm-main} and \ref{Thm-variance}.

\begin{corollary} \label{cor-normalized-moment-as}
Under the assumptions of Theorem \ref{Thm-main}, for every fixed $p \in \bN_+$,
\begin{align} \label{eq-example-normalized-moment}
    \frac1{n^k} \bE \left[ \Tr \widetilde M_{n,k,m}^p \right]
    &= \Bigg[ \sum_{s=1}^p c^s \sum_{\alpha \in \cC_{s,p}^{(1)}}
    \left( \prod_{a=1}^s m_{\deg_a(\alpha)}^{(\tau)} \right)
    \left( \prod_{a=1}^s \frac{H_n(\deg_a(\alpha))}{n^{\deg_a(\alpha)}} \right)^k \Bigg] (1+o(1)).
\end{align}
Moreover,
\begin{align} \label{eq-example-normalized-moment-as}
    \frac1{n^k} \Tr \widetilde M_{n,k,m}^p
    &= \Bigg[ \sum_{s=1}^p c^s \sum_{\alpha \in \cC_{s,p}^{(1)}}
    \left( \prod_{a=1}^s m_{\deg_a(\alpha)}^{(\tau)} \right)
    \left( \prod_{a=1}^s \frac{H_n(\deg_a(\alpha))}{n^{\deg_a(\alpha)}} \right)^k \Bigg] (1+o(1))
\end{align}
almost surely, simultaneously for all fixed $p \in \bN_+$.
\end{corollary}

\begin{proof}
Since $\widetilde M_{n,k,m} = n^{-k} M_{n,k,m}$ and $\sum_{a=1}^s \deg_a(\alpha) = p$, formula \eqref{eq-example-normalized-moment} follows from Theorem \ref{Thm-main}.
By Theorem \ref{Thm-variance},
\begin{align*}
    \Var \left( \frac1{n^k} \Tr \widetilde M_{n,k,m}^p \right)
    &= n^{-2pk} \Var \left( \frac1{n^k} \Tr M_{n,k,m}^p \right)
    \le \left( \frac{C_p}{n} \right)^k.
\end{align*}
For all sufficiently large $n$, we have $k \ge 2$ and $C_p/n<1$. Therefore,
\begin{align*}
    \sum_{n=1}^{\infty} \Var \left( \frac1{n^k} \Tr \widetilde M_{n,k,m}^p \right)
    &< \infty.
\end{align*}
By the Chebyshev's inequality and the Borel--Cantelli lemma, we establish that
\begin{align} \label{eq-example-centered-moment-as}
    \frac1{n^k} \Tr \widetilde M_{n,k,m}^p
    - \frac1{n^k} \bE \left[ \Tr \widetilde M_{n,k,m}^p \right]
    &\longrightarrow 0
    \qquad \text{almost surely}.
\end{align}

The term corresponding to $s=p$ and $\alpha=(1,\ldots,p)$ in the sum in \eqref{eq-example-normalized-moment} is $c^p(m_1^{(\tau)})^p > 0$, since $H_n(1)=n$. All summands in \eqref{eq-example-normalized-moment} are nonnegative, so the sum is bounded away from zero. Combining \eqref{eq-example-normalized-moment} and \eqref{eq-example-centered-moment-as} yields \eqref{eq-example-normalized-moment-as}. Finally, taking the intersection of these probability-one events over $p \in \bN_+$ concludes the proof.
\end{proof}

\subsection[The case m4=1]{The case $m_4=1$}
\label{sec:example-unit-modulus}

In this subsection, we consider the special setting that $m_4=1$.
Then $\bE [(|\xi_1|^2 - 1)^2] = m_4 - 1 = 0$, which implies that $|\xi_1|=1$ almost surely.
The following conclusions hold without any restrictions on the growth rate of $k$.

\begin{proposition} \label{prop-example-unit-modulus}
Under the assumptions of Theorem \ref{Thm-main}, if $|\xi_1|=1$ almost surely, then, for every fixed $p \in \bN_+$,
\begin{align} \label{eq-example-unit-modulus-moments}
    \frac1{n^k} \bE \left[ \Tr \widetilde M_{n,k,m}^p \right]
    &\longrightarrow \sum_{s=1}^p c^s \sum_{\alpha \in \cC_{s,p}^{(1)}} \prod_{a=1}^s m_{\deg_a(\alpha)}^{(\tau)}, \qquad n \to \infty.
\end{align}
The random moment $n^{-k} \Tr \widetilde M_{n,k,m}^p$ converges to the same limit almost surely.
\end{proposition}

\begin{proof}
Lemma \ref{lem-norm-moment-polynomial} gives
\begin{align*}
    H_n(p)
    &= n^p, \qquad \forall p \in \bN_+.
\end{align*}
The proof is concluded by substituting this identity into Corollary \ref{cor-normalized-moment-as}.
\end{proof}

\begin{corollary} \label{Coro-example-unit-modulus}
Under the assumptions of Proposition \ref{prop-example-unit-modulus}, suppose in addition that $m_q^{(\tau)} \le A^q q^q$ for all $q \in \bN_+$ and some constant $A > 0$.
Then the ESD of $\widetilde M_{n,k,m}$ converges weakly, both in expectation and almost surely, to a uniquely determined probability measure.
\end{corollary}

\begin{proof}
By Lemma \ref{lem-Bai}, there are at most $4^p$ canonical non-crossing sequences of length $p$. Since

\begin{align*}
    \prod_{a=1}^s m_{\deg_a(\alpha)}^{(\tau)}
    &\le A^p \prod_{a=1}^s \deg_a(\alpha)^{\deg_a(\alpha)} \le A^p p^p,
\end{align*}
the limiting $p$th moment in \eqref{eq-example-unit-modulus-moments} is at most $[4 A \max\{1,c\}]^p p^p$.
Thus Carleman's condition is satisfied, and the limiting moment sequence uniquely determines a probability measure.
The proof is concluded.
\end{proof}

\begin{remark}
In the case that $\tau_\alpha \equiv 1$, it is verified in \cite{Yuan2024} that this unique probability measure is exactly the Mar\v{c}enko--Pastur law.
\end{remark}

\begin{remark}
Corollary \ref{Coro-example-unit-modulus} recovers convergence of ESD of $\widetilde M_{n,k,m}$ in \cite{Yuan2024}*{Corollary 1.1} under the present assumptions. The corresponding moment results are given in \cite{Yuan2024}*{Theorems 1.1 and 1.2}.
\end{remark}

\subsection[The case k/n=o(1)]{The case $k/n=o(1)$}
\label{sec:example-small-k}

In the setting $k/n=o(1)$, Corollary \ref{cor-normalized-moment-as} implies the following result.

\begin{proposition} \label{prop-example-small-k-moments}
Under the assumptions of Theorem \ref{Thm-main}, if $k/n \to 0$, then for every fixed $p \in \bN_+$,
\begin{align} \label{eq-example-small-k-moments}
    \frac1{n^k} \bE \left[ \Tr \widetilde M_{n,k,m}^p \right]
    &\longrightarrow \sum_{s=1}^p c^s \sum_{\alpha \in \cC_{s,p}^{(1)}} \prod_{a=1}^s m_{\deg_a(\alpha)}^{(\tau)}.
\end{align}
The random moment $n^{-k} \Tr \widetilde M_{n,k,m}^p$ converges to the same limit almost surely.
\end{proposition}

\begin{proof}
In order to use Corollary \ref{cor-normalized-moment-as}, we need to estimate the quantity $H_n(d)/n^d$ first.

For $d=1$, $H_n(1)/n = 1$, and for every fixed $d \ge 2$,
\begin{align} \label{eq-first-norm-correction}
    \frac{H_n(d)}{n^d}
    &= 1 + \binom d2 \frac{m_4 - 1}{n} + O(n^{-2}).
\end{align}
Indeed, in the expansion of $\bE \left[ \| \widetilde{\y} \|^{2d} \right] = n^{-d} \bE \left[ (\sum_{b=1}^n |\xi_b|^2)^d \right]$, the terms with $d$ distinct coordinate indices contribute
\begin{align*}
    \frac{(n)_d}{n^d}
    &= 1 - \binom d2 \frac1n + O(n^{-2}).
\end{align*}
The terms with exactly $d-1$ distinct indices have one index occurring twice and contribute
\begin{align*}
    \binom d2 m_4 \frac{(n)_{d-1}}{n^d}
    &= \binom d2 \frac{m_4}{n} + O(n^{-2}).
\end{align*}
All remaining terms use at most $d-2$ distinct indices and contribute $O(n^{-2})$. This proves \eqref{eq-first-norm-correction}.

Hence, for every fixed $d \ge 1$,
\begin{align} \label{eq-example-first-log-correction}
    \log \frac{H_n(d)}{n^d}
    &= \frac{m_4 - 1}{n} \binom d2 + O(n^{-2}),
\end{align}
which implies
\begin{align*}
    k \sum_{a=1}^s \log \frac{H_n(\deg_a(\alpha))}{n^{\deg_a(\alpha)}}
    &= O(k/n) = o(1).
\end{align*}
Taking the exponential and substituting into \eqref{eq-example-normalized-moment} and \eqref{eq-example-normalized-moment-as}, we obtain the convergence in expectation and almost surely.
\end{proof}

Similar to Corollary \ref{Coro-example-unit-modulus}, an application of Carleman's condition yields the following corollary immediately.

\begin{corollary} \label{Coro-example-small-k-moments}
Under the assumptions of Proposition \ref{prop-example-small-k-moments}, suppose in addition that $m_q^{(\tau)} \le A^q q^q$ for all $q \in \bN_+$ and some constant $A > 0$.
Then the ESD of $\widetilde M_{n,k,m}$ converges weakly, both in expectation and almost surely, to a uniquely determined probability measure.
\end{corollary}

\begin{remark}
Corollary \ref{Coro-example-small-k-moments} gives convergence in expectation and almost surely in the setting of \cite{BYY2022}*{Proposition 2.4}.
We would like to point out that the convergence of the ESD of $\widetilde M_{n,k,m}$ was first obtained in \cite{Lytova2018}*{Theorem 1.2 and Remark 4.1} using the Stieltjes transform method.
\end{remark}

\subsection[The case k/n=O(1)]{The case $k/n=O(1)$}
\label{sec:example-linear-k}

Now we turn to the regime $k/n=O(1)$ in which $k$ can grow much faster. We have the following result.

\begin{proposition} \label{prop-example-linear-k-moments}
Under the assumptions of Theorem \ref{Thm-main}, if $k/n=O(1)$, then for every fixed $p \in \bN_+$,
\begin{align} \label{eq-example-linear-k-asymptotic}
    \frac1{n^k} \bE \left[ \Tr \widetilde M_{n,k,m}^p \right]
    &= \sum_{s=1}^p c^s \sum_{\alpha \in \cC_{s,p}^{(1)}} \left( \prod_{a=1}^s m_{\deg_a(\alpha)}^{(\tau)} \right) \exp \left\{ \frac{k}{n} (m_4 - 1) \sum_{a=1}^s \binom{\deg_a(\alpha)}2 \right\} + o(1).
\end{align}
If, in addition, $k/n \to d \in [0,\infty)$, then
\begin{align} \label{eq-example-linear-k-limit}
    \lim_{n \to \infty} \frac1{n^k} \bE \left[ \Tr \widetilde M_{n,k,m}^p \right]
    &= \sum_{s=1}^p c^s \sum_{\alpha \in \cC_{s,p}^{(1)}} \left( \prod_{a=1}^s m_{\deg_a(\alpha)}^{(\tau)} \right) \exp \left\{ d (m_4 - 1) \sum_{a=1}^s \binom{\deg_a(\alpha)}2 \right\}.
\end{align}
Both formulas also hold for $n^{-k} \Tr \widetilde M_{n,k,m}^p$ almost surely.
\end{proposition}

\begin{proof}
For fixed $p$, by \eqref{eq-example-first-log-correction}, we obtain
\begin{align*}
    k \sum_{a=1}^s \log \frac{H_n(\deg_a(\alpha))}{n^{\deg_a(\alpha)}}
    &= \frac{k}{n} (m_4 - 1) \sum_{a=1}^s \binom{\deg_a(\alpha)}2 + O(k/n^2),
\end{align*}
for all $\alpha \in \cC_{s,p}^{(1)}$.
Taking the exponential and applying Corollary \ref{cor-normalized-moment-as}, we obtain \eqref{eq-example-linear-k-asymptotic} and its almost sure version, noting that $O(k/n^2) = O(n^{-1})$. 
If $k/n \to d$, taking the limit gives \eqref{eq-example-linear-k-limit} and its almost sure version.
\end{proof}

\begin{remark}
For $d>0$, the limiting moment formula \eqref{eq-example-linear-k-limit} agrees with \cite{BYY2022}*{Theorem 2.1(i)} under the present assumptions.
\end{remark}

\begin{remark}
For the case $d=0$, Proposition \ref{prop-example-linear-k-moments} reduces to Proposition \ref{prop-example-small-k-moments}.
\end{remark}

\subsection[The case k/n tending to infinity]{The case $k/n\to\infty$}
\label{sec:example-quadratic-k}

In this subsection, we consider $m_4 > 1$ and $k/n \to \infty$, without an upper bound on the growth of $k$.
In particular, the conclusions below apply when $k$ is of order $n^a$ for any fixed real number $a > 1$.

\begin{proposition} \label{prop-example-moment-growth}
Under the assumptions of Theorem \ref{Thm-main}, suppose that $m_4 > 1$ and $k/n \to \infty$. For every fixed integer $p \ge 2$,
\begin{align}
    \frac1{n^k} \bE [\Tr \widetilde M_{n,k,m}^p]
    &= c m_p^{(\tau)} \left( \frac{H_n(p)}{n^p} \right)^k (1 + o(1)) \label{eq-example-general-moment-growth} \\*
    &= c m_p^{(\tau)} \exp \left\{ \left[ \binom p2 (m_4 - 1) + O(n^{-1}) \right] \frac{k}{n} \right\} (1 + o(1)). \label{eq-example-general-n-power}
\end{align}
In particular, if $k/n^a \to 1$ for a fixed real number $a > 1$, then
\begin{align} \label{eq-example-polynomial-k-growth}
    \frac1{n^k} \bE [\Tr \widetilde M_{n,k,m}^p]
    &= c m_p^{(\tau)} \exp \left\{ \left[ \binom p2 (m_4 - 1) + o(1) \right] n^{a-1} \right\}.
\end{align}
All three formulas also hold for $n^{-k} \Tr \widetilde M_{n,k,m}^p$ almost surely.
\end{proposition}

\begin{proof}
In \eqref{eq-example-normalized-moment} and \eqref{eq-example-normalized-moment-as}, the unique term with $s=1$ corresponds to $\alpha = (1,\ldots,1)$, and is
\begin{align} \label{eq-dominated-1}
    c m_p^{(\tau)} \left( \frac{H_n(p)}{n^p} \right)^k.
\end{align}
For $s \ge 2$ and $\alpha \in \cC_{s,p}^{(1)}$, we have $\sum_{a=1}^s \deg_a(\alpha) = p$. By \eqref{eq-example-first-log-correction},
\begin{align*}
    \log \frac{\prod_{a=1}^s H_n(\deg_a(\alpha))}{H_n(p)}
    &= \log \left( \frac{\prod_{a=1}^s H_n(\deg_a(\alpha))}{n^p} \frac{n^p}{H_n(p)} \right) \\
    &= \frac{m_4 - 1}{n} \left[ \sum_{a=1}^s \binom{\deg_a(\alpha)}2 - \binom p2 \right] + O(n^{-2}) \\
    &= - \frac{m_4 - 1}{n} \sum_{1 \le a < b \le s} \deg_a(\alpha) \deg_b(\alpha) + O(n^{-2}).
\end{align*}
Since
\begin{align*}
    \sum_{1 \le a < b \le s} \deg_a(\alpha) \deg_b(\alpha)
    &\ge \deg_1(\alpha) (p - \deg_1(\alpha)) \ge p - 1,
\end{align*}
for all sufficiently large $n$,
\begin{align*}
    \left( \frac{\prod_{a=1}^s H_n(\deg_a(\alpha))}{H_n(p)} \right)^k
    &\le \exp \left\{ - (m_4 - 1)(p - 1) \frac{k}{n} + O(k/n^2) \right\} \longrightarrow 0.
\end{align*}
By Lemma \ref{lem-Bai}, the number of summands is finite and depends only on $p$. Thus, the term \eqref{eq-dominated-1} is dominant.
Corollary \ref{cor-normalized-moment-as} gives \eqref{eq-example-general-moment-growth} and its almost sure version.

Using \eqref{eq-example-first-log-correction} once more,
\begin{align*}
    k \log \frac{H_n(p)}{n^p}
    &= \binom p2 (m_4 - 1) \frac{k}{n} + O(k/n^2).
\end{align*}
Taking the exponential proves \eqref{eq-example-general-n-power} and its almost sure version. If $k/n^a \to 1$, then $k/n = (1 + o(1)) n^{a-1}$ and $O(k/n^2) = o(n^{a-1})$, which yields \eqref{eq-example-polynomial-k-growth} and its almost sure version.
\end{proof}

From Proposition \ref{prop-example-moment-growth}, if $k/n^a \to 1$ for some fixed $a > 1$, then for any $p \in \bN_+$ with $p \ge 2$, the $p$th moment grows faster than every fixed power of $n$, both in expectation and almost surely. In particular, the largest eigenvalue of $\widetilde M_{n,k,m}$ is unbounded.

Next, we turn to the estimate of the largest eigenvalue of $\widetilde M_{n,k,m}$ with the help of these moment estimates. We provide a lower bound in the following corollary.

\begin{corollary} \label{cor-example-quadratic-top}
Under the assumptions of Proposition \ref{prop-example-moment-growth},
\begin{align} \label{eq-example-quadratic-top-lower}
    \liminf_{n \to \infty} \frac{\lambda_{\max}(\widetilde M_{n,k,m})}{(1 + (m_4 - 1)/n)^k}
    &\ge \frac{m_2^{(\tau)}}{m_1^{(\tau)}} \qquad \text{almost surely},
\end{align}
where $\lambda_{\max}$ denotes the largest eigenvalue.
\end{corollary}

\begin{proof}
Corollary \ref{cor-normalized-moment-as} with $p=1$ gives
\begin{align} \label{eq-example-quadratic-first-as}
    \frac1{n^k} \Tr \widetilde M_{n,k,m}
    &\longrightarrow c m_1^{(\tau)} > 0 \qquad \text{almost surely}.
\end{align}
By \eqref{eq-H-degree-recursion}, $H_n(2)/n^2 = 1 + (m_4 - 1)/n$. Proposition \ref{prop-example-moment-growth} with $p=2$ therefore yields
\begin{align} \label{eq-example-quadratic-second-liminf}
    \frac{n^{-k} \Tr \widetilde M_{n,k,m}^2}{(1 + (m_4 - 1)/n)^k}
    &\longrightarrow c m_2^{(\tau)} \qquad \text{almost surely}.
\end{align}
Since $\widetilde M_{n,k,m}$ is positive semi-definite, the eigenvalues of $\widetilde M_{n,k,m}$ is non-negative, and hence,
\begin{align*}
    \lambda_{\max}(\widetilde M_{n,k,m})
    &\ge \frac{\Tr \widetilde M_{n,k,m}^2}{\Tr \widetilde M_{n,k,m}},
\end{align*}
where the denominator is positive for all sufficiently large $n$ almost surely by \eqref{eq-example-quadratic-first-as}.
Combining this inequality with \eqref{eq-example-quadratic-first-as} and \eqref{eq-example-quadratic-second-liminf} proves \eqref{eq-example-quadratic-top-lower}.
\end{proof}

\begin{remark}
In particular, \eqref{eq-example-quadratic-top-lower} implies
\begin{align*}
    \liminf_{n \to \infty} \frac{n}{k} \log \lambda_{\max}(\widetilde M_{n,k,m})
    &\ge m_4 - 1 > 0
\end{align*}
almost surely.
\end{remark}

The Cauchy--Schwarz inequality applied to the eigenvalues gives, for every $p \ge 2$,
\begin{align*}
    \frac{\Tr \widetilde M_{n,k,m}^p}{\Tr \widetilde M_{n,k,m}^{p-1}}
    &\le \frac{\Tr \widetilde M_{n,k,m}^{p+1}}{\Tr \widetilde M_{n,k,m}^p},
\end{align*}
whenever the denominators are positive.
Thus, using higher order moments gives a better lower bound for the largest eigenvalue. More precisely, we have the following improvement.

\begin{corollary} \label{cor-example-higher-order-top}
Under the assumptions of Proposition \ref{prop-example-moment-growth},
\begin{align*}
    \frac{n}{k} \log \lambda_{\max}(\widetilde M_{n,k,m})
    &\longrightarrow +\infty \qquad \text{almost surely}.
\end{align*}
\end{corollary}

\begin{proof}
Fix an integer $p\ge3$. The non-negativity of the eigenvalues of $\widetilde M_{n,k,m}$ yields
\begin{align} \label{eq-example-higher-order-trace-bound}
    \lambda_{\max}(\widetilde M_{n,k,m})
    &\ge \frac{\Tr \widetilde M_{n,k,m}^p}{\Tr \widetilde M_{n,k,m}^{p-1}},
\end{align}
whenever the denominator is positive.
By Proposition \ref{prop-example-moment-growth}, the denominator is eventually positive almost surely, and
\begin{align} \label{eq-converge-ratio-moment}
    \frac{\Tr \widetilde M_{n,k,m}^p}{\Tr \widetilde M_{n,k,m}^{p-1}} \left( \frac{n H_n(p-1)}{H_n(p)} \right)^k
    &\longrightarrow \frac{m_p^{(\tau)}}{m_{p-1}^{(\tau)}} \qquad \text{almost surely}.
\end{align}
Combining \eqref{eq-converge-ratio-moment} and \eqref{eq-example-higher-order-trace-bound}, we have
\begin{align} \label{eq-example-higher-order-top-lower}
    \liminf_{n \to \infty} \frac{\lambda_{\max}(\widetilde M_{n,k,m})}{\left( \frac{H_n(p)}{n H_n(p-1)} \right)^k}
    &\ge \frac{m_p^{(\tau)}}{m_{p-1}^{(\tau)}} \qquad \text{almost surely}.
\end{align}
By \eqref{eq-example-first-log-correction},
\begin{align} \label{eq-example-higher-order-log}
    \log \frac{H_n(p)}{n H_n(p-1)}
    &= \frac{(p-1)(m_4 - 1)}{n} + O(n^{-2}).
\end{align}
By \eqref{eq-example-higher-order-top-lower}, almost surely for all sufficiently large $n$,
\begin{align*}
    \frac{n}{k} \log \lambda_{\max}(\widetilde M_{n,k,m})
    &\ge n \log \frac{H_n(p)}{n H_n(p-1)} + \frac{n}{k} \log \frac{m_p^{(\tau)}}{2 m_{p-1}^{(\tau)}}.
\end{align*}
Using \eqref{eq-example-higher-order-log} and $n/k \to 0$, we obtain
\begin{align*}
    \liminf_{n \to \infty} \frac{n}{k} \log \lambda_{\max}(\widetilde M_{n,k,m})
    &\ge (p-1)(m_4 - 1) \qquad \text{almost surely}.
\end{align*}
As $m_4 > 1$, the proof is concluded by letting $p \to \infty$.
\end{proof}

\section*{Acknowledgments}

WY was supported by Guangdong Basic and Applied Basic Research Foundation (No. 2026A1515030040), and National Natural Science Foundation of China (No. 12501183), and a Grant of the Department of Science and Technology of Guangdong Province (No. 2024QN11X161).

\section*{Statements and Declarations}

All the figures are generated with AI assistance.

\bibliographystyle{plain}
\bibliography{tensor}

\end{document}